\documentclass[11pt]{article}

\usepackage[T1]{fontenc}
\usepackage[utf8]{inputenc}

\usepackage{amsmath,amssymb,amsfonts,amsthm,mathtools}

\usepackage{graphicx}
\usepackage{float}

\usepackage{booktabs}
\usepackage{array}
\usepackage{multirow}

\usepackage[colorlinks=true,
            linkcolor=blue,
            citecolor=blue,
            urlcolor=blue]{hyperref}

\usepackage{enumitem}

\usepackage{xcolor}
\usepackage{verbatim}

\usepackage[a4paper,margin=1in]{geometry}
\usepackage{cancel}
\usepackage{yhmath}

\newtheorem{theorem}{Theorem}[section]

\newtheorem{proposition}[theorem]{Proposition}
\newtheorem{corollary}[theorem]{Corollary}

\theoremstyle{definition}
\newtheorem{definition}[theorem]{Definition}
\newtheorem{example}[theorem]{Example}
\newtheorem{remark}[theorem]{Remark}

\newcommand{\RR}{\mathbb{R}}
\newcommand{\RF}{\ensuremath{\mathbb{R}_\mathcal{F}}}
\newcommand{\RFA}{\ensuremath{\mathbb{R}_{\mathcal{F}(A)}}}
\newcommand{\CC}{\mathbb{C}}
\newcommand{\CF}{\ensuremath{\mathbb{C}_\mathcal{F}}}
\newcommand{\CFA}{\ensuremath{\mathbb{C}_{\mathcal{F}(A)}}}
\newcommand{\FP}{\ensuremath{\mathcal{FP}}}

\title{On Fuzzy Partial Differential Equations for A-Linearly Interactive Fuzzy Complex Processes: Sobolev Spaces and Fourier Analysis}

\author{
Silvio Antonio Bueno Salgado\thanks{Institute of Applied Social Sciences, Federal University of Alfenas, Varginha, Minas Gerais, Brazil. Email: \texttt{silvio.salgado@unifal-mg.edu.br}. Corresponding author.}
\and
Estev\~ao Esmi\thanks{Department of Applied Mathematics, State University of Campinas, Campinas, S\~ao Paulo, Brazil. Email: \texttt{eesmi@unicamp.br}.}
\and
Francielle Santo Pedro Sim\~oes\thanks{Multidisciplinary Department, Federal University of S\~ao Paulo, Osasco, S\~ao Paulo, Brazil. Email: \texttt{fsimoes@unifesp.br}.}
\and
La\'ecio Carvalho de Barros\thanks{Department of Applied Mathematics, State University of Campinas, Campinas, S\~ao Paulo, Brazil. Email: \texttt{laecioba@unicamp.br}.}
}

\date{}

\begin{document}

\maketitle

\begin{abstract}
This article develops a mathematical framework to handle fuzzy partial differential equations (PDEs) using Sobolev spaces defined over $A$-linearly interactive complex fuzzy processes. We introduce the notion of weak derivative in the fuzzy sense to define fuzzy Sobolev spaces that preserve key analytical properties. Furthermore, we introduce a fuzzy Fourier transform adapted to this context and explore its main properties. Applications to fuzzy versions of the heat and Schrödinger equations are presented to demonstrate the effectiveness and generality of the proposed framework. This approach not only extends classical tools to the fuzzy setting, but also provides new insights into the treatment of uncertainty in dynamic systems.
\end{abstract}

\medskip
\noindent\textbf{Keywords.}
Fuzzy Sobolev spaces; Fuzzy Fourier transform;
$A$-linearly interactive complex fuzzy processes;
Fuzzy partial differential equations.

\section{Introduction}\label{sec1}

Many problems that arise naturally from physics and other sciences can be described by partial differential equations (PDE). Sobolev spaces provide a suitable framework for studying the existence, uniqueness, and regularity of solutions to PDEs, such as the Heat and Schrödinger equations. 
Another essential tool in the analysis of PDEs is the Fourier transform, as it can characterize functions in Sobolev spaces and aids in deriving fundamental solutions to PDEs \cite {Tanabe}. 
Roughly speaking, the study of solutions of PDEs using the Fourier transform follows three steps. The first step consists of applying the Fourier transform to convert the PDE in the spatial domain to an ordinary differential equation (ODE) in the frequency domain. 
This is possible because of a key property of the Fourier transform: turning derivatives into multiplication. 
The second step corresponds to obtaining a solution for this ODE, which is generally a complex-valued function. In the third and last step, the solution of the original PDE is given by applying the inverse Fourier transform to the solution of the corresponding ODE. In this context, Sobolev spaces are the natural function spaces in which such solutions are defined.

In many cases, when modeling with PDEs, certain parameters or quantities of the phenomenon under study may be uncertain, and a possible approach to address this uncertainty is to use fuzzy set theory \cite{Zadeh}. Buckley and Feuring \cite{BF} introduced the concept of partial differential equations as a way to model uncertainty present in certain phenomena in a dynamical environment. 
Some studies on fuzzy Fourier transform have been presented. Bede el al. \cite{Bedee} studied Fourier transforms of fuzzy-number-valued functions and applications in fingerprint coding. Gouyandeh et al. \cite{Gou} studied an analytical solution of a fuzzy heat equation under the generalized Hukuhara partial differentiability by fuzzy Fourier transform. Shi et al. \cite{Shi} studied the fuzzy Fourier transform, which is based on Henstock-Kurzweil integrals on infinite interval. 

Sobolev spaces are based on the notion of weak derivative, which arises as a generalization of integration by parts. 
Thus, defining this space in a fuzzy setting requires appropriate fuzzy algebraic and calculus tools that allow the establishment of a version of the integration by parts rule for fuzzy functions. 
Most approaches to fuzzy calculus do not have a suitable integration-by-parts formula, and the main reason for this is that, in general, multiplication and addition on fuzzy numbers do not satisfy the distributivity property.    
In contrast, the class of $A$-linearly interactive fuzzy numbers is a two-dimensional Banach space isomorphic to the Euclidean space $\mathbb{R}^2$ \cite{Esmi} that satisfies the distributive with respect to the notion of the $A$-cross product \cite{bia}. 
In  \cite{FranIntegral}, the authors established an integration-by-parts formula for functions taking values in the class of $A$-linearly interactive fuzzy numbers with respect to the notion of the $A$-cross product. 

Since we are interested in solving FPDEs using the Fourier transform, we must consider a fuzzy version of complex Sobolev spaces. 
Salgado {\it et al.} introduced the notion of $A$-linearly interactive complex fuzzy numbers, which is a generalization of $A$-linearly interactive fuzzy numbers to the complex setting \cite{Salgado}. In the same article, the authors proved that this space is a two-dimensional complex Banach space that is isomorphic to $\mathbb{C}^2$, in which the set of complex numbers is embedded.
Despite the fact that the notion of the $A$-cross product is already defined for this space, to the best of our knowledge no integration-by-parts formula has yet been established in this setting. This article presents such a formula for functions taking values in the set of $A$-linearly interactive complex fuzzy numbers, allowing us to define fuzzy complex Sobolev spaces.
Based on these spaces, we study fuzzy partial differential equations using the Fourier Transform in this setting. 
We apply the proposed approach to find solutions for the fuzzy heat equation and the fuzzy Schr\"odinger equation. Furthermore, the uniqueness of the solutions is proved using the appropriate Sobolev spaces for each problem.

The paper is organized as follows. Section 2 presents the space of A-linearly interactive complex fuzzy numbers, denoted here by $\mathbb{C}_{\mathcal{F}(A)}$.
Section 3 introduces Sobolev spaces for A-linearly interactive complex fuzzy  processes. In Section 4, we present a fuzzy version of the Fourier transform in this setting, and prove several of its properties. In Section 5, we apply this fuzzy Fourier transform to find weak solutions to the fuzzy heat equation and the fuzzy Schr\"odinger equation. We show that these solutions are unique by considering the appropriate Sobolev spaces for each problem. The final conclusions drawn from the present study are given in Section 6.

\section{Mathematical Background}

Let $A$ be a fuzzy subset of $\RR$. The degree of membership of $x\in \RR$ in $A$ is denoted by $A(x)$ and the $\alpha$-level set of $A$ by $[A]_\alpha.$
We say that $A$ is a fuzzy (real) number if, for all $\alpha \in [0,1]$, $[A]_\alpha$ is a non-empty, closed, and bounded interval of $\RR.$
We denote the set of all fuzzy numbers by $\RF$, and the $\alpha$-level set of $A\in \RF$, alternatively, by $[\underline{a}_\alpha, \overline{a}_\alpha].$

The fuzzy number $A$ is called asymmetric (or non-symmetric) if, for all $y\in\mathbb{R}$, there exists some $x_0\in\RR$ such that $A(x_0-y) \neq A(x_0+y)$. The class of asymmetric fuzzy numbers forms a dense and open subset of $\mathbb{R}_{\mathcal{F}}$ with respect to the Hausdorff metric \cite{SLI}. This implies that any fuzzy number can be arbitrarily approximated by an asymmetric fuzzy number, and that asymmetry is preserved under small perturbations. 

A complex fuzzy number $Z$ can be expressed in the form $C+Di$, where $C$ and $D$ are fuzzy real numbers and $i$ is the imaginary unit \cite{Fu}. We denote the set of all complex fuzzy numbers by $\CF$. Similarly to the classical case, standard arithmetic operations on $\CF$ can be defined in terms of arithmetic operations on $\RF$.  Let $Z = (C+Di) \in \CF $, $W = (A+Bi) \in \CF$, and $\omega = \lambda +i \theta \in \CC$, the addition of $Z$ and $W$ and the multiplication of $\omega$ and $Z$ are defined, respectively, by 
\begin{equation}\label{eq:standar_addition}
Z + W = (C+A) + i(D+B) 
\end{equation}
and 
\begin{equation}\label{eq:standar_scalar_product}
\omega Z = \left(\lambda C -\theta D\right) + i\left(\theta C + \lambda D\right), 
\end{equation}
where $A,B,C,D \in \RF$ and $\lambda, \theta \in \RR.$
Note that $C+A$ and $D+B$ stand for the standard sum of $C$ and $A$, and of $D$ and $B.$  
One can easily verify the following distributive law holds:
\begin{equation}\label{eq:distributive_law}
\omega(Z + W) = \omega Z + \omega W.     
\end{equation}
 
For each $A \in \mathbb{R}_{\mathcal{F}}$, consider the operator $\psi: \CC^2 \longrightarrow \CF$ given by 
\begin{equation}
\psi(u,v)=u+vA,
\end{equation}
for every $u,v \in \CC$ \cite{Salgado}.
We denote the range of the operator $\psi$ by $\CFA$, and each element of $\CFA$ is called a $A$-linearly interactive complex fuzzy number.
Note that the set of complex numbers $\CC$ is embedded in $\CFA$ since every complex number $u$ can be identified with the complex fuzzy number $\psi(u,0) \in \CFA$, and we simply denote this fact by $\mathbb{C} \subset \CFA$ \cite {Salgado}.
The following theorem states that $\CFA$ is isomorphic to $\CC^2$ if $A$ is asymmetric. 

\begin{theorem} \cite{Salgado} \label{cor:banach}
Let $A \in \RF$ be non-symmetric. The space $\CFA$ is a two-dimensional vector space over the scalar field $\CC$, with the standard scalar multiplication given in Equation~\eqref{eq:standar_scalar_product} and the addition $\oplus$ defined by
\begin{equation}\label{eq:sum_cfa}
Z \oplus W= \psi\left( \psi^{-1}(Z)+ \psi^{-1}(W)\right)    
\end{equation}
for all $Z,W \in \CFA.$
Furthermore, $\CFA$ is isomorphic to $\CC^2$ through the linear isomorphism $\psi.$ 

If $\|\cdot\|$ is a norm on $\CC^2$, then the mapping $B\mapsto \|\psi^{-1}(B)\|$, for all $B\in \CFA$, defines a norm on $\CFA$ induced by the isometry $\psi.$ 
\end{theorem}

From now on, we assume that $\CFA$ is a Banach space, i.e., the underlying fuzzy number $A$ is asymmetric. 
Let $Z = (u + vA), W = (z + wA) \in \CFA$, with $u,v,z,w \in \CC.$ 
Since $\psi^{-1}(Z) = (u,v)$ and $\psi^{-1}(W) = (z,w)$, Equation~\eqref{eq:sum_cfa} can be rewritten in terms of the sum of the coefficients $u,v,z,w$ as follows: 
\begin{equation}\label{eq:sum_cfa_b}
Z \oplus W= (u+z) + (v+w)A\,.    
\end{equation}
In addition, the opposite of $Z$ is $-Z =(-1)Z = -u -vA \in \CFA$, and the difference of $Z$ and $W$ is defined by 
\[
Z\ominus W = Z \oplus (-W) = (u-z) + (v-w)A\,.
\]  

The notion of the $A$-cross product on $\CFA$ was introduced in \cite{Salgado} for the case where the 1-level set of $A$ is a singleton. In this article, we present a more general version of this concept by giving up the aforementioned restriction on the 1-level set of $A$. To this end, let us first consider the midpoint of the 1-level set of a fuzzy number $B$ given by $m(B) = 0.5(\underline{b}_1 + \overline{b}_1) \in \RR$, and   
the midpoint of 1-level set of a complex fuzzy number $W = C + iD$, with $C,D \in \RF$, is given by $m(W) = m(C+iD)= m(C) + i\, m(D) \in \CC.$ 

\begin{definition} \cite{Salgado} \label{cross_C}  
Let $A \in \RF$ be asymmetric. 
For $Z,W \in \CFA$, we define the $A$-cross product 
of $Z$ and $W$ by 
\begin{equation}\label{eq:eq4}
Z \odot W = m(W)Z \oplus m(Z)W \ominus m(Z)m(W)\,.   
\end{equation}
\end{definition}

For $Z = u + vA$, $W = z + wA$, and $a=m(A)$, the $A$-cross product of $Z$ and $W$ can be rewritten in terms of the complex numbers $u,v,z,w$ and the real number $a$ as follows:  
\begin{eqnarray}\label{eq:A_product}
\nonumber Z\odot W & =& (u+va)(z + wA) \oplus  (z+wa)(u + vA) \ominus (u+va)(z+wa) \\ \nonumber
& =& ((u+va)z + (z+wa)u - (u+va)(z+wa)) + ((u+va)w+(z+wa)v)A \\
& =& (zu - a^2vw) + (uw + zv + 2avw)A.
\end{eqnarray}

The following proposition states several interesting properties that the $A$-cross product. 

\begin{proposition}
Let $A$ be an asymmetric fuzzy number. 
For every $Z,W,Y \in \CFA$ and $\omega \in \CC$, it holds:
\begin{itemize}
\item[1.] $Z\odot W = W\odot Z$;
\item[2.] $Z\odot (\omega W) = \omega (Z\odot W)$;
\item[3.] $Z\odot (W \oplus Y) = (Z\odot W) \oplus (Z\odot Y)$;
\item[4.] $Z\odot (W\odot Y) = (Z\odot W)\odot Y$.
\end{itemize}
\end{proposition}
\begin{proof}
Despite the fact that Definition~\ref{cross_C} is slightly different from that presented in \cite{bia}, the proofs of properties 1. -- 4. follows the same ideas and steps found in \cite{bia}. 
\end{proof}

The restriction of the operator to $\RR\times \RR$ leads to the space $\RFA$ introduced in \cite{Esmi}. More precisely, the space of $\RFA$ is given by the range of 
\[
\tilde{\psi} = \psi\Big\vert_{\RR\times \RR}.
\] 
$\RFA$ is a two-dimensional real vector space that is isomorphic to $\RR^2$ via the isomorphism $\tilde{\psi}$ whenever $A$ is asymmetric~\cite{Esmi}.  

Similarly to the classical case, one can note that the real Banach space $\RFA$ can be embedded in the complex Banach space $\CFA$ 
by considering the mappings $B\mapsto B+i0$ and $\lambda \mapsto \lambda+i0$ for all $B \in \RFA$ and for all $\lambda \in \RR$. 
Thus, the vector addition and the scalar product on $\RFA$ can be obtained from the addition and the scalar multiplication given in Eqs~\eqref{eq:standar_addition} and \eqref{eq:standar_scalar_product} on $\CFA$ as follows: 
\begin{equation}\label{eq:operations_Rfa}
B\oplus C = (B + i0) \oplus (C + i0) \quad and \quad \lambda B = (\lambda + i0)(B + i0).    
\end{equation}
Using these connections, we can easily verify that the operations on $\RFA$ above coincide with those induced by the bijection $\tilde{\psi}$.
Furthermore, the composition of these mappings with the $A$-cross product, together with the assumption that $[A]_1 = \{a\}$, yields the notion of the $A$-cross product on $\RFA$ \cite{bia}: 
\begin{eqnarray*}
B\odot C & = & m(C)B \oplus m(B)C \ominus m(B)m(C) \\
& = & \left(rs -a^2pq\right)+\left(2pqa+rq+sp\right)A \in \RFA
\end{eqnarray*}
for $B=(r + pA)\in \RFA$ and $C = (s+qA) \in \RFA$.

In the same way that $\CC$ is related to $\RR^2$, the 
complex space $\CFA$ can be uniquely associated with $\RFA^2$ by the mapping $\psi$:
\[
\psi(p+ir,\,q+is)=(p+ir) +(q+is)A=\underbrace{(p+qA)}_{\in \RFA}+i\underbrace{(r+sA)}_{\in \RFA}\,.
\] 
This implies that the mapping $(B,C) \mapsto B + iC$ is a bijective operator from $\RFA^2$ to $\CFA$, which allows us to rewrite the arithmetic operations on $\CFA$ in terms of the operations on $\RFA$. 
Given $Z = B+iD \in \CFA$ and $W=C + iE \in \CFA$, with $B,C,D,E\in \RFA$, we have: 
\begin{itemize}
    \item $Z\oplus W = (B\oplus C) + i(D\oplus E)$;
    \item $Z\oplus W = (B\odot C \ominus D\odot E) + i(B\odot E \oplus C\odot D)$.    
\end{itemize}

In view of Theorem~\ref{cor:banach}, any norm on $\CC^2$ induces a norm on $\CFA$ through the bijection $\psi.$ Recall that for any non singular matrix $M \in \CC^{4\times 4}$ and the p-norm $\|\cdot\|_p$, the function $\|M\mathbf{w}\|_p$, for all $\mathbf{w}\in \CC^2$, is a norm on $\CC^2.$ Thus, from Theorem~\ref{cor:banach}, the mapping $\|Z\|_{\tilde{p}} = \|M\psi^{-1}(Z)\|_p$, for all $Z\in \CFA$, is a norm on $\CFA.$ If $M$ is the identity matrix, then we simply denote the norm  
$\|\cdot\|_{\tilde{p}}$ on $\CFA$ by $\|\cdot\|_{p}.$ 

Considering the connections between the spaces $\RFA$ and $\CFA$ and the norm on $\RFA$ provided in \cite{frederico2025fuzzy}, in this paper, we focus on the norm $\|\cdot\|_{\tilde{p}}$ with the matrix $M$ given by 
\begin{equation}\label{eq:M}
M = \begin{bmatrix}
1 & a \\ 0 & 1
\end{bmatrix}
\end{equation}
where $a = m(A).$ So, for every $B = u + vA \in \CFA$, we consider the following norm
\begin{eqnarray}\label{eq:pnormCFA}
\nonumber \|B\|_{\tilde{p}} & = & \left\|M\psi^{-1}(B) \right\|_{p} \\
\nonumber  & = & \left\|\begin{bmatrix}
1 & a \\ 0 & 1
\end{bmatrix}
\begin{bmatrix}
u \\ v
\end{bmatrix}
\right\|_{p} \\
 & = & \left\lbrace \begin{array}{cl}
 \left( \lvert u + av\rvert^p + \lvert v\rvert^p \right)^{\frac{1}{p}} &, \mbox{ if } p \in [1,\infty)  \\
 \max\{\lvert u + av\rvert,\lvert v\rvert\}     & , \mbox{ if } p = \infty
 \end{array}  \right.\,.
\end{eqnarray}

A well-known result is that any two norms on a finite-dimensional Banach spaces are equivalent. Therefore, any two norms on $\CFA$ are equivalent. The following corollary states a useful equivalence between two specific norms on $\CFA$ that we will use throughout this text. 

\begin{corollary}\label{cor:equivalent_norms}
The norm $\|\cdot \|_{\tilde{p}}$, given in Equation~\eqref{eq:pnormCFA}, and the norm $\|\cdot\|_\infty$, given by $\| u + vA\|_\infty = \max\{\lvert u \rvert, \lvert v \rvert\}$ for all $u,v \in \CC$, are equivalent norms on $\CFA.$    
\end{corollary}

Another interesting result that will be useful for the proposal of this work is the inequality stated in the following proposition. 

\begin{proposition}\label{prop:inequality_norm_1}
For $B,C \in \CFA$, we have
\begin{equation}\label{eq:desigualdade_norma_produto}
 \|B\odot C \|_{\tilde{1}} \leq  \|B\|_{\tilde{1}} \|C\|_{\tilde{1}} \,. 
\end{equation}
\end{proposition}
\begin{proof}
The proof follows the same steps as the proof of Proposition 2(i) presented in \cite{frederico2025fuzzyII} for the space $\RFA.$
\end{proof}

The following section introduces the concept of $A$-linearly interactive complex fuzzy processes, as well as the associated differential and integral calculus developed for them.

\section{Differential and Integral Calculus on $\CFA$}

Let $\Omega$ be a nonempty subset of the Euclidean space $\RR^n$, $n\geq1$. A function $F \colon \Omega \longrightarrow \CFA$ is said to be 
an $A$-linearly interactive fuzzy complex process when $\CFA$ is regarded as the Banach space $(\CFA,\oplus, \cdot, \|\cdot\|_{\tilde{p}}).$  
We denote the set of all $A$-linearly interactive fuzzy complex processes defined on $\Omega$ by $\FP(\Omega,\CFA)$. 
Since $\RR^n$ and $\CFA$ are Banach spaces, the notions of bounded, convergence, limit, continuity,  etc. are defined in a natural way for $A$-linearly interactive fuzzy complex processes. In particular, the notions of derivative \cite{Esmi,Shen} and partial derivative \cite{de2021differential,Mina} were first introduced for the space $\RFA$ and then extended to $\CFA$ in \cite{Salgado}.

\begin{definition}
Let $\|\cdot\|$ be a norm on $\RR^n$ and $F \in \FP(\Omega,\CFA)$, with $\Omega \subseteq \RR^n$, $n\geq 1$. 
\begin{itemize}

\item[(a)] $F$ is bounded if there exists $L > 0$ such that $\|F(\omega)\|_{\tilde{p}} \leq L$ for all $\omega \in \Omega;$ 

\item[(b)] $F$ is continuous at $\omega \in \Omega$ if for every $\epsilon > 0$ there exists $\delta > 0$ such that $\|F(\omega) - F(z)\|_{\tilde{p}} < \epsilon$ for all $z \in \Omega$ satisfying $\|\omega - z\| < \delta$;

\item[(c)] For $n=1$ and $\omega \in \mathrm{int}(\Omega)$, we say that $F$ is differentiable\footnote{Some authors refer to this notion of differentiability as $A$-differentiability or $\psi$-differentiability (see e.g. \cite{Salgado}). However, since we do not use any other notion of differentiability in this work, we choose not to include these prefixes for the sake of simplicity.}
at $\omega$ if there exists $F'(\omega) \in \CFA$, called derivative of $F$ at $\omega$, that satisfies  
\begin{equation}
\lim_{h\to 0} \left\|F'(\omega) \ominus  \frac{F(\omega+h) \ominus F(\omega)}{h} \right\|_{\tilde{p}} = 0.   
\end{equation}
The derivative of order $k\geq 0$, denoted by $F^{(k)}$ or $\frac{\partial^{k}}{\partial \omega^{k}}F$, is defined recursively with $F^{(0)} = F$. 

\item[(d)] For $n>1$, $\omega = (\omega_1,\ldots,\omega_n) \in \mathrm{int}(\Omega)$, and $i \in \{1,\ldots,n\}$, we say that $F$ is $i$th- partial differentiable at $\omega$ if there exists $\frac{\partial}{\partial \omega_i}F(\omega) \in \CFA$, called $i$th-partial derivative of $F$ at $\omega$, that satisfies  
\begin{equation}
\lim_{h\to 0} \left\|\frac{\partial}{\partial \omega_i}F(\omega) \ominus  \frac{F(\omega+he_i) \ominus F(\omega)}{h} \right\|_{\tilde{p}} = 0.   
\end{equation}
where $e_i$ is $i$th vector of the canonical basis of $\RR^n$. 
The derivative of order $k\geq 0$, denoted by $\frac{\partial^{k}}{\partial \omega_i^{k}}F$, is defined recursively with $\frac{\partial^{0}}{\partial \omega_i^{0}}F = F$. 

\item[(e)] Given a multi-index $\alpha = (\alpha_1,...,\alpha_n)$, that is, $\alpha_i \in \mathbb{N}_0 := \mathbb{N} \cup \{0\}$ for all $i=1,\ldots,n$ and $\lvert \alpha\rvert = \sum_{i=1}^n \alpha_i \leq k$. Consider the differential operator $D^\alpha$ defined by 
\begin{equation}\label{eq:OperatorD_alpha}
D^\alpha F = \frac{\partial^{\mid\alpha\mid}}{\partial \omega_1^{\alpha_1}\cdots\partial \omega_n^{\alpha_n}}F = 
\frac{\partial^{\alpha_1}}{\partial \omega_1^{\alpha_1}}\left(\frac{\partial^{\alpha_2}}{\partial \omega_2^{\alpha_2}}\left( \ldots \left(\frac{\partial^{\alpha_n}}{\partial \omega_n^{\alpha_n}} F(\omega) \right) \right) \right)\,,    
\end{equation}
\end{itemize}
provided the partial derivatives exist. 
Based on this definition, we consider the following subsets of $\FP(\Omega,\CFA)$:
\begin{itemize}

\item[$\bullet$] $C^k(\Omega,\CFA)$ denotes the set of all $F \in \FP(\Omega,\CFA)$ that are $k$ times continuously (partially) differentiable, that is, 
$D^\alpha F(\omega)$ exists at every
$\omega \in \mathrm{int}(\Omega)$ and multi-index $\alpha$, $\lvert\alpha\rvert\leq k$, and is continuous. 

\item[$\bullet$] $C_c^k(\Omega,\CFA)$ denotes the subset of $C^k(\Omega,\CFA)$ composed of \\ $F \in C^k(\Omega,\CFA)$ such that $\operatorname{supp}(F) := \overline{\{t \in \Omega \mid F(t) \neq 0\}}$, called the support of $F$, is compact, where $\overline{Y}$ stands for the closure of $Y$ relative to $\Omega$.

\item[$\bullet$] $C_c^\infty(\Omega,\CFA)$ represents the set of infinitely differentiable functions with compact support. The elements of $C_c^\infty(\Omega,\CFA)$ are also called \emph{test functions}.

\end{itemize}

\end{definition}

From Theorem~\ref{cor:banach}, for every $F \in \FP(\Omega,\CFA)$ there exist unique functions $u, v \colon \Omega \longrightarrow \CC$ such that 
\[
F(\omega) = u(\omega) + v(\omega)A,\quad \forall \omega \in \Omega.
\]
In addition, it follows that $(u,v) = \psi^{-1}\circ F$.
This reveals a bijection between the set $\FP(\Omega,\CFA)$ and the set of pairs of complex-valued functions defined on $\Omega$, denoted here by $\mathcal{F}(\Omega,\CC^2)$. 
As an immediate consequence of this connection,  $\FP(\Omega,\CFA)$ forms a complex vector space, with the corresponding operations defined pointwise. For $F,G \in \FP(\Omega,\CFA)$ and $\lambda \in \CC$, 
the sum of $F$ and $G$ and the product $\lambda$ and $F$ are the $A$-linearly interactive fuzzy complex processes $F\oplus G$ and $\lambda F$ given, respectively, by 
\[
(F\oplus G)(\omega) = F(\omega) \oplus G(\omega) \quad \mbox{ and } \quad 
(\lambda F)(\omega) = \lambda F(\omega)
\]
for all $\omega \in \Omega.$

In the same way that we use the map $\psi$ to carry algebraic and topological structures from $\CC^2$ to $\CFA$, we can use the bijection above (which is also based on $\psi$) to define concepts and function spaces of interest composed of elements of $\FP(\Omega,\CFA).$   

\begin{definition}\label{def:integral_CFA}
Let $F \in \FP(\Omega,\CFA)$ be given by $F(\omega) = u(\omega) + v(\omega)A$ for all $\omega \in \Omega$, where $u,v:\Omega\to \CC$ and  $\Omega \subseteq \RR^n$, $n\geq 1$.  

\begin{itemize}
\item[(a)] We say that $F$ is (Riemann/Lesbesgue) integrable in $\Omega$ if $u$ and $v$ are (Riemann/Lesbesgue) integrable\footnote{Some authors refer to this notion of integral as $A$-integral or $\psi$-integral (see e.g. \cite{Salgado}). However, since we do not use any other notion of integral in this work, we choose not to include these prefixes for the sake of simplicity.}\footnote{Usually, the notion of Riemann integral for $A$-linearly interactive fuzzy (complex) processes defined in an interval $[a,b]$ is derived from the well-established definition of Riemann integral for mappings between Banach spaces \cite{santo2020calculus,Zeidler}. However, in the context of the spaces $\CFA$ or $\RFA$, it exists iff the Riemann integrals in $[a,b]$ of a pair of classical functions exist, which provides an alternative way to define integrability in these spaces. Here, we adopt this alternative formulation since it can be used to define a more general notion of Lebesgue integrability in a region $\Omega$ of $\RR^n$. This same approach was employed in \cite{son2021fractional} to define fractional operators and the Laplace transform in $\RFA$.} in $\Omega$. In this case, the (Riemann/Lesbesgue) integral of $F$ in $\Omega$ is given by 
\begin{equation}
\int_\Omega F(\omega) d\omega  = \left(\int_\Omega u(\omega) d\omega\right) + \left(\int_\Omega v(\omega) d\omega\right)A\,.   
\end{equation}

\item[(b)] The conjugate of $F$ is $F^* \in \FP(\Omega,\CFA)$ given by 
\begin{equation}
 F^*(\omega) = u(\omega)^* + v(\omega)^*A
\end{equation}
where $u(\omega)^*$ and $v(\omega)^*$ are the complex conjugates of the complex numbers $u(\omega)$ and $v(\omega)$, respectively.
\end{itemize}

\end{definition}

\begin{remark}
Note that item (a) of Definition~\ref{def:integral_CFA}  also encompasses the notion of improper integrals since $\Omega$ may be unbounded. For example, if $F:[a,\infty)\to \CFA$ is integrable, then 
\begin{equation}
\int_{a}^{\infty}{F(\omega) d\omega}  :=\lim_{b\to\infty}{\int_{a}^{b}{F(\omega) d\omega}} = \left(\lim_{b\to\infty}{\int_{a}^{b}{u(\omega) d\omega}}\right) +  \left(\lim_{b\to\infty}{\int_{a}^{b}{v(\omega) d\omega}}\right)A.
\end{equation}
where $F(\omega) = u(\omega) + u(\omega)A$ for all $\omega.$
\end{remark}

Every $F \in \FP(\Omega,\CFA)$ can be decomposed in order to highlight its structure in terms of its real and imaginary components. 
If $u,v:\Omega\to \CC$ is given by $(u,v) = \psi^{-1}\circ F$, then 
$F(\omega) = u(\omega)+v(\omega)A$ for all $\omega \in \Omega.$ Since the codomain of $u$ and $v$ is the set of complex numbers, there exist functions $r,p,s,q:\Omega \to \RR$ such that 
\[
u(\omega)=r(\omega) + i p(\omega)\;\mbox{ and } \; v(\omega)=s(\omega) + i q(\omega)\,,\forall \omega \in \Omega.
\]
Thus, we can express $F$ as follows:
\begin{eqnarray*}
F(\omega) & = & u(\omega) + u(\omega)A \\ 
& = & (r(\omega) + i p(\omega)) + (s(\omega) + i q(\omega)A \\ 
& = & \underbrace{(r(\omega) + s(\omega)A)}_{:= F_1(\omega)} \;\oplus \;i\underbrace{(p(\omega) + q(\omega)A)}_{:= F_2(\omega)} \\
& = & \underbrace{F_1(\omega)}_{\in \RFA} \; \oplus \; i\underbrace{F_2(\omega)}_{\in \RFA} 
\end{eqnarray*}
Thus, setting $F_1(\omega) =r(\omega)+ s(\omega)A \in \RFA$ and $F_2(\omega) = p(\omega)+q(\omega)A \in\RFA$, we obtain
\begin{equation}\label{functionF}
F(\omega) = F_1(\omega) \; \oplus \; iF_2(\omega)\,,\quad \forall \omega \in \Omega.
\end{equation}
This decomposition clarifies the real and imaginary contributions of $F(t)$ and matches the structural representation of elements of $\FP(\Omega,\CFA)$ as pairs of $\RFA$-valued functions. Consequently, theoretical and computational results developed for $\RFA$ can be directly adapted to $\CFA$.

In the literature, the results of the following proposition is stated either with respect to the norm $\|\cdot\|_\infty$ on $\CFA$ or within the setting of $\RFA$. Ne\-vertheless, due to the equivalence of norms (see Corollary~\ref{cor:equivalent_norms}) and the structural relation between $\CFA$ and $\RFA^2$, the result remains valid when the norm $\|\cdot\|_{\tilde{p}}$ is used.

\begin{proposition}\label{prop:Basic_calculus_property}
Let $F \in \FP(\Omega,\CFA)$ be given by $F(\omega) = u(\omega) + v(\omega)A$ for all $\omega \in \Omega \subseteq \RR^n$, $n\geq 1.$ 
\begin{itemize}

\item[(a)] $F$ is bounded if, and only if, $u,v:\Omega \to \CC$ are bounded; 

\item[(b)] $F$ is continuous if, and only if, $u,v:\Omega \to \CC$ are continuous; 

\item[(c)] $F$ is $k$ times (partially) differentiable at $\omega \in \mathrm{int}(\Omega)$ if, and only if, $u,v: \Omega \to \CC$ are $k$ times (partially) differentiable at $\omega \in \mathrm{int}(\Omega)$. Moreover, 
\begin{equation}
    \frac{\partial^k}{\partial \omega_i^k}F(\omega) = \frac{\partial^k}{\partial \omega_i^k}u(\omega) + \frac{\partial^k}{\partial \omega_i^k}v(\omega)A\,,
\end{equation}
for $i\in \{1,\ldots,n\}.$ In particular, for $n=1$, it reduces to 
\begin{equation}
F^{(k)}(\omega) = u^{(k)}(\omega) + v^{(k)}(\omega)A. 
\end{equation}

\item[(d)] If $F$ is $i$th partially differentiable at $\omega$, then $F$ is continuous at $\omega$ with respect to the $i$th argument. For $n=1$, it reduces to the statement: if $\exists F'(\omega)$, then $F$ is continuous at $\omega;$

\item[(e)] Let $G \in \FP(\Omega,\CFA)$ and $\lambda \in \CC$. If $F$ and $G$ are $i$th-paritally differentiable at $\omega$, then 
\[
\frac{\partial}{\partial \omega_i}(F \oplus \lambda G)(\omega) = \frac{\partial}{\partial \omega_i}F(\omega) \;\oplus\; \lambda\frac{\partial}{\partial \omega_i}G(\omega)\,.
\]
For $n=1$, it reduces to 
\begin{equation*}
(F\oplus \lambda G)(\omega) = F'(\omega) \oplus \lambda G'(\omega). 
\end{equation*}

\item[(f)] Fuzzy Clairaut-Schwarz theorem: If $F \in C^k(\Omega,\CFA)$ and $\omega \in \mathrm{int}(\Omega)$, then for any multi-index $\alpha$, $\lvert \alpha \rvert \leq k$, and for all permutation $j:\{1,\ldots,n\}\to \{1,\ldots,n\}$, we have 
\[
D^\alpha F(\omega) = \frac{\partial^{\alpha_{j(1)}}}{\partial \omega_{j(1)}^{\alpha_{j(1)}}} \ldots \frac{\partial^{\alpha_{j(n)}}}{\partial \omega_{j(n)}^{\alpha_{j(n)}}} F(\omega) \,.
\]

\item[(g)] Let $G \in \FP(\Omega,\CFA)$ and $\lambda \in \CFA$. If $F$ and $G$ are integrable in $\Omega$, then 
\[
\int_\Omega (F \oplus \lambda \odot G)(\omega) d\omega = 
\int_\Omega F(\omega) d\omega \; \oplus \; \lambda \odot\int_\Omega G(\omega) d\omega\,.  
\]

\item[(h)] If $F$ is continuous and $\Omega$ a compact set of $\RR^n$, then 
$F$ is integrable in $\Omega$.

\item[(i)] Fuzzy Fubini–Tonelli theorem: If $\Omega = [a_1,b_1]\times\ldots\times[a_n,b_n]$ and $F$ is integrable in $\Omega$, then for any permutation $j:\{1,\ldots,n\}\to \{1,\ldots,n\}$ we have 
\[
\int_{a_1}^{b_1}\ldots \int_{a_n}^{b_n} F(\omega) \,d\omega_{1}\ldots,d\omega_{n} = \int_{a_{j(1)}}^{b_{j(1)}}\ldots \int_{a_{j(n)}}^{b_{j(n)}} F(\omega) \,d\omega_{j(1)}\ldots,d\omega_{j(n)}\,.
\]

\item[(j)] If $\Omega = [a,b]$ and $F$ is integrable on $\Omega$, then the $A$-linearly interactive complex fuzzy process $G$ defined by 
\begin{equation}
    G(t)= \int_{a}^{t}F(s)ds
\end{equation}
is differentiable at every $t \in (a,b)$ and 
\begin{equation}
    G'(t)=F(t).
\end{equation}  

\item[(k)] If $\Omega = [a,b]$ and $F$ is  integrable on $\Omega$, then the $A$-linearly interactive complex fuzzy process $G$ defined by 
\begin{equation}
    G(t)= \int_{a}^{t}F(s)ds
\end{equation}
is differentiable at every $t \in (a,b)$ with $G'(t)=F(t).$

\item[(l)] If $\Omega = [a,b]$ and $F$ is continuously differentiable in $\Omega$, then 
\begin{equation}
\int_{a}^{b}F'(s)ds = F(b) \ominus F(a). 
\end{equation}

\item[(m)] If $F$ is integrable in $\Omega$, then 
\begin{equation}
\left\|\int_{\Omega}F(s)ds\right\|_{\tilde{1}} \leq  \int_{\Omega}\left\|F(s)\right\|_{\tilde{1}}ds. 
\end{equation}

\item[(m)] Let $G:[a,b] \times \Omega \to \CFA.$ If $G(t,\cdot):\Omega \to \CFA$ is integrable in $\Omega$ for all $t \in [a,b]$ and $\frac{\partial}{\partial t} G$ exists 
\begin{equation}
\left\|\int_{\Omega}F(s)ds\right\|_{\tilde{1}} \leq  \int_{\Omega}\left\|F(s)\right\|_{\tilde{1}}ds. 
\end{equation}

\end{itemize}
\end{proposition}
\begin{proof}
Items (a) and (b) follow immediately from Corollary~\ref{cor:equivalent_norms}. Item (c) was proved in \cite{Salgado}. Items (d)–(f) are consequences of item (c), together with the corresponding results from the calculus of complex-valued functions. Items (h)–(m) follow from Definition~\ref{def:integral_CFA}, item (c), and the corresponding results from the calculus of complex-valued functions.
It remains to prove item (g). To this end, consider $a=m(A)$, 
$F=u+vA$, $G=z+wA$, and $\lambda = \lambda_1 + \lambda_2A$, with $u,v,z,w:\Omega\to \CC$ and $\lambda_1,\lambda_2 \in \CC$. 
From Equation~\eqref{eq:A_product} and Definition~\ref{def:integral_CFA}, we have 
\begin{align*}
& \int_\Omega (F \oplus \lambda \odot G)(\omega) d\omega = \\
& \; = \int_\Omega (u(\omega)+v(\omega)A) \oplus (\lambda_1 + \lambda_2A) \odot (z(\omega) + w(\omega)A) d\omega \\
& \; = \int_\Omega (u(\omega)+ \lambda_1 z(\omega) -a^2\lambda_2 w(\omega)) \,d\omega + \\
& \quad + A\int_\Omega (v(\omega)+ \lambda_1 w(\omega) +\lambda_2 z(\omega) +2a\lambda_2 w(\omega)) \,d\omega \\
& \; = \int_\Omega u(\omega) d\omega + \lambda_1 \int_\Omega z(\omega)d\omega -a^2\lambda_2  \int_\Omega w(\omega) d\omega + \\
& \quad + A\left(\int_\Omega v(\omega)d\omega + \lambda_1 \int_\Omega  w(\omega)d\omega +\lambda_2 \int_\Omega z(\omega)d\omega +2a\lambda_2 \int_\Omega w(\omega))d\omega\right) \\
& \; = {\small \left(\int_\Omega u(\omega) d\omega + A \int_\Omega v(\omega) d\omega\right) \oplus   (\lambda_1 + \lambda_2A)\odot\left(\int_\Omega z(\omega)d\omega + A\int_\Omega w(\omega)d\omega\right)}
\\ &
= \int_\Omega F(\omega) d\omega \; \oplus \; \lambda \odot\int_\Omega G(\omega) d\omega\,.      
\end{align*}
\end{proof}

One can relate the concept of integrability in $\CFA$ to that in $\RFA$ \cite{FranIntegral,Shen} by considering the real and imaginary decomposition given in Equation~\eqref{functionF}.
Let $F(\omega)  =  u(\omega) + u(\omega)A$, with 
$u(\omega)=r(\omega) + i p(\omega)$ and $v(\omega)=s(\omega) + i q(\omega)$, for all $\omega \in \Omega.$ We can rewrite $F$ as $F(\omega) = F_1(\omega) \; \oplus \; iF_2(\omega)$, where $F_1(\omega) =r(\omega)+ s(\omega)A \in \RFA$ and $F_2(\omega) = p(\omega)+q(\omega)A \in\RFA$ for all $\Omega$. 
If $F$ is partially differentiable at some $\omega \in (a,b)$, then, from item (c) of Proposition~\ref{prop:Basic_calculus_property}, we obtain that the functions $r,s,p,q:\Omega \to \RR$ are partially differentiable at $\omega$. This implies that $A$-linearly interactive fuzzy processes $F_1$ and $F_2$ are differentiable at $\omega$ \cite{Esmi,Mina,Shen}: 
\begin{equation}\label{FlinhaporPsi1}
\frac{\partial}{\partial \omega_i}F(\omega) = \frac{\partial}{\partial \omega_i}F_1(\omega) + i \frac{\partial}{\partial \omega_i}F_2(\omega),
\end{equation}
where $\frac{\partial}{\partial \omega_i}F_1(\omega)$ and $\frac{\partial}{\partial \omega_i}F_2(\omega)$ belong to $\RFA$. 
Similarly, if $F$ is continuous, bounded, or  integrable in $\Omega = [a,b]$, then $F_1$ and $F_2$ are also continuous, bounded, or  integrable in $[a,b]$. In the last case, we have 
\[
\int_a^b F(\omega)d\omega = \int_a^b F_1(\omega)d\omega \;\oplus \; i\int_a^b F_2(\omega)d\omega\,.
\]

The next proposition establishes that the (Leibniz) product rule and integration-by-parts formulas can be derived for $A$-linearly interactive fuzzy complex processes with respect to the $A$-cross product.

\begin{proposition}\label{partes}
  Let $F,G \in \FP(\Omega,\CFA)$ with $\Omega \subset \RR^n$. The following properties hold true:
    \begin{enumerate}
        \item[a.] If $F$ and $G$ are continuous in $\Omega$, then $F\odot G$ is continuous in $\Omega$.
        \item[b.] If $F$ and $G$ are $i$th partially differentiable at $\omega \in \mathrm{int}(\Omega)$, then $F\odot G$ is $i$th partially differentiable at $\omega$, and 
        \begin{equation}\label{eq:regraderivadaproduto}
        \frac{\partial}{\partial \omega_i}(F(\omega)\odot G(\omega))=\left(\frac{\partial}{\partial \omega_i}F(\omega)\odot G(\omega)\right)\oplus \left(F(\omega)\odot \frac{\partial}{\partial \omega_i}G(\omega)\right).
        \end{equation}

        \item[c.] If $F$ and $G$ are integrable in $\Omega$, then $F\odot G$ is  integrable in $\Omega$.
        
        \item[d.] If $\Omega = [a,b]$ and $F, G \in C^1(\Omega,\CFA)$, then 
        \begin{equation}\label{eq:int_by_parts}
           \int_a^b{(F'(\omega)\odot G(\omega)) d\omega} = (F(\omega)\odot G(\omega))\Big\vert_a^b \ominus \int_a^b{(F(\omega)\odot G'(\omega)) d\omega}.
        \end{equation}

        \item[e.] If $\Omega$ is an interval of $\RR$, $F \in C^k(\Omega,\CFA)$ and $\Phi \in C_{c}^{\infty}(\Omega,\CFA)$, then  
        \begin{equation}\label{eq:parts_generalized}
        \int_{\Omega} F(\omega) \odot \Phi^{(k)}(\omega) \; d\omega = (-1)^{\mid \alpha \mid} \int_{\Omega} F^{(k)}(\omega) \odot \Phi(\omega) \; d\omega.
        \end{equation}

        \item[f.] If $\Omega = \RR^n$, $F \in C^k(\Omega,\CFA)$ and $\Phi \in C_{c}^{\infty}(\Omega,\CFA)$, then for all multi-index $\alpha$, $\lvert \alpha\rvert \leq k$, we have 
        \begin{equation}\label{eq:parts_generalized}
        \int_{\Omega} F(\omega) \odot D^{\alpha} \Phi(\omega) \; d\omega = (-1)^{\mid \alpha \mid} \int_{\Omega} D^{\alpha} F(\omega) \odot \Phi(\omega) \; d\omega.
        \end{equation}
        
    \end{enumerate}
\end{proposition}
\begin{proof}
Let $F(\omega) = u(\omega) + v(\omega)A$ and $G(\omega) = z(\omega) + w(\omega)A$ for all $\omega \in \Omega$, with $u,v,z,w:\Omega\to \CC.$
By Equation~\eqref{eq:A_product}, we obtain 
{\small
\begin{eqnarray}\label{eq:aux_rules}
(F\odot G)(\omega) & = &  F(\omega) \odot G(\omega) \\ \nonumber
 & = &  (z(\omega)u(\omega) - a^2v(\omega)w(\omega)) + (u(\omega)w(\omega) + z(\omega)v(\omega) + 2av(\omega)w(\omega))A    
\end{eqnarray}}
for all $\omega \in \Omega$, where $a = m(A)$.

\begin{enumerate}
\item[a.] From item (b) of Proposition~\ref{prop:Basic_calculus_property}, the functions $u,v,z,w$ are continuous in $\Omega.$ This implies that the right-hand side of Equation~\eqref{eq:aux_rules} is also continuous since it is given by compositions of continuous functions. Again, from item (b) of Proposition~\ref{prop:Basic_calculus_property}, we conclude that $F\odot G$ is continuous in $\Omega.$

 \item[b.] The proof presented here follows the same steps as those provided in \cite{FranIntegral} for the space $\RFA$.  
 From item (c) of Proposition~\ref{prop:Basic_calculus_property}, the functions $u,v,z,w$ are $i$th partially differentiable at $\omega \in \mathrm{int}(\Omega)$. Using the classical product rule for complex differentiation in the right-hand side of Equation~\eqref{eq:aux_rules}, we obtain 
{\small
\begin{align*}
\frac{\partial}{\partial \omega_i}&(F\odot G)(\omega)  =   \\
= &  \frac{\partial}{\partial \omega_i}\left(z(\omega)u(\omega) - a^2v(\omega)w(\omega) + (u(\omega)w(\omega) + z(\omega)v(\omega) + 2av(\omega)w(\omega))A\right) \\
= &  \left(\frac{\partial}{\partial \omega_i}z(\omega)\right)u(\omega) + z(\omega)\frac{\partial}{\partial \omega_i}u(\omega) - a^2 \left(\frac{\partial}{\partial \omega_i}v(\omega)\right)w(\omega) -a^2v(\omega)\frac{\partial}{\partial \omega_i}w(\omega) \\
& +\left[\left(\frac{\partial}{\partial \omega_i}u(\omega)\right)w(\omega) + u(\omega)\frac{\partial}{\partial \omega_i}w(\omega) + \left(\frac{\partial}{\partial \omega_i}z(\omega)\right)v(\omega) \right. \\
& + \left. z(\omega)\frac{\partial}{\partial \omega_i}v(\omega) + 2a\left(\frac{\partial}{\partial \omega_i}v(\omega)\right)w(\omega) + 2av(\omega)\frac{\partial}{\partial \omega_i}w(\omega) \right]A \\
= & \left[ z(\omega)\frac{\partial}{\partial \omega_i}u(\omega) - a^2 \left(\frac{\partial}{\partial \omega_i}v(\omega)\right)w(\omega) + \left( \left(\frac{\partial}{\partial \omega_i}u(\omega)\right)w(\omega) +  z(\omega)\frac{\partial}{\partial \omega_i}v(\omega) \right. \right. \\
&  + \left.\left. 2a\left(\frac{\partial}{\partial \omega_i}v(\omega)\right)w(\omega)\right)A\right] \oplus \left[ \left(\frac{\partial}{\partial \omega_i}z(\omega)\right)u(\omega)  -a^2v(\omega)\frac{\partial}{\partial \omega_i}w(\omega) \right. \\
& \left.\left( u(\omega)\frac{\partial}{\partial \omega_i}w(\omega) + \left(\frac{\partial}{\partial \omega_i}z(\omega)\right)v(\omega) + 2av(\omega)\frac{\partial}{\partial \omega_i}w(\omega) \right)A\right] \\
= & \left(\frac{\partial}{\partial \omega_i}F(\omega)\odot G(\omega)\right)\oplus \left(F(\omega)\odot \frac{\partial}{\partial \omega_i}G(\omega)\right).
\end{align*}}

\item[c.] By Definition~\ref{def:integral_CFA}(a), the functions $u,v,z,w$ are integrable in $\Omega.$ Using the fact that the product of two complex-valued integrable functions is also integrable, we obtain that the right-hand side of Equation~\eqref{eq:aux_rules} is integrable in $\Omega$. Therefore, the integral of $F\odot G$ in $\Omega$ exists. 

\item[d.] Since fuzzy versions of the fundamental calculus theorem (see Proposition~\ref{prop:Basic_calculus_property}(i)) and of the Leibniz product rule (i.e., item b. above) are established for $A$-linearly interactive fuzzy complex processes, we can prove the integration-by-parts formula following the same steps of the classical settings. On the one hand, we have 
\[
\int_a^b (F\odot G)'(\omega) d\omega = 
F(b)\odot G(b) \ominus F(a)\odot G(a) =
F(\omega)\odot G(\omega)\Big\vert_a^b\,.
\]
On the other hand, from the linearity of the integral and the Leibniz product rule, we have 
\[
\int_a^b (F\odot G)'(\omega) d\omega = 
\int_a^b F'(\omega)\odot G(\omega)d\omega \;\oplus \; \int_a^b F(\omega)\odot G'(\omega) d\omega\,.
\]
Combining these both equalities, we obtain the 
integration-by-parts formula given by Equation~\eqref{eq:int_by_parts}.    

\item[e.] Since $\Phi \in C_{c}^{\infty}(\Omega,\CFA)$, there exists $ [a,b] \subset \Omega$ such that 
$\Phi(\omega) = 0$ for $\omega \leq a$ or $\omega \geq b$. This implies that $\Phi^{(j)}(\omega) = 0$ for all $\omega \leq a$ or $\omega \geq b$ and for all integer $j\leq 0$. Hence, $\Phi^{(j)} \in C_{c}^{\infty}(\Omega,\CFA)$ for every $j\geq 0$.  Using the integration-by-parts formula above, we obtain
\begin{eqnarray*}
\int_{a}^{b} F(\omega) \odot \Phi^{(k)}(\omega) \; d\omega & = &
\int_{a}^{b} F(\omega) \odot \frac{\partial}{\partial \omega} \left(\Phi^{(k-1)} (\omega)\right) \; d\omega \\ 
& = & \cancelto{0}{F(\omega)\odot \Phi^{(k-1)}(\omega)\Big\vert_a^b} \ominus \int_{a}^{b} F'(\omega) \odot \Phi^{(k-1)} (\omega) \; d\omega \\
& = & -1\left(\cancelto{0}{F'(\omega)\odot \Phi^{(k-2)}(\omega)\Big\vert_a^b} \ominus \int_{a}^{b} F''(\omega) \odot \Phi^{(k-2)} (\omega) \; d\omega  \right) \\
& \vdots & \\
& = & (-1)^k\int_{a}^{b} F^{(k)}(\omega) \odot \Phi(\omega) \; d\omega\,. 
\end{eqnarray*}

\item[f.] Since $\Phi \in C_{c}^{\infty}(\RR^n,\CFA)$, there exists $C= [a_1,b_1]\times\ldots\times[a_n,b_n] \subset \RR^n$ such that 
$\Phi(\omega) = 0$ for $\omega \in D = \RR^n\setminus C$. This implies that $D^\alpha \Phi(\omega) = 0$ for all $\omega \in D$ and multi-index $\alpha$. Hence, $D^\alpha \Phi \in C_{c}^{\infty}(\RR^n,\CFA)$ for every multi-index $\alpha$.  

Given $\alpha$ be a multi-index such that $\lvert \alpha \rvert \leq k$, from item e. above and the fuzzy Clairaut-Schwarz and Fubini–Tonelli theorems (see Proposition~\ref{prop:Basic_calculus_property}), we have
{\small
\begin{align*}
& \int_{\Omega} F(\omega) \odot D^{\alpha}(\omega) \Phi \; d\omega = \\ 
&  =
\int_{a_n}^{b^n}\ldots \left[\int_{a_1}^{b^1} F(\omega) \odot \frac{\partial^{\alpha_1}}{\partial \omega_1^{\alpha_1}} \left(\frac{\partial^{\alpha_2}}{\partial \omega_2^{\alpha_2}} \ldots \frac{\partial^{\alpha_n}}{\partial \omega_n^{\alpha_n}}  \Phi(\omega) \right) \; d\omega_1\right]\ldots d\omega_n \\
& =
(-1)^{\alpha_1}\int_{a_n}^{b^n}\ldots \int_{a_1}^{b^1}\left[\int_{a_2}^{b^2} \frac{\partial^{\alpha_1}}{\partial \omega_1^{\alpha_1}}F(\omega) \odot  \frac{\partial^{\alpha_2}}{\partial \omega_2^{\alpha_2}} \ldots \frac{\partial^{\alpha_n}}{\partial \omega_n^{\alpha_n}}  \Phi(\omega) \; d\omega_2\right]d\omega_1\ldots d\omega_n \\
&  =
(-1)^{\alpha_1+\alpha_2}\int_{a_n}^{b^n}\ldots \int_{a_1}^{b^1}\left[\int_{a_2}^{b^2} \frac{\partial^{\alpha_2}}{\partial \omega_2^{\alpha_2}}\frac{\partial^{\alpha_1}}{\partial \omega_1^{\alpha_1}}F(\omega) \odot  \frac{\partial^{\alpha_3}}{\partial \omega_3^{\alpha_3}} \ldots \frac{\partial^{\alpha_n}}{\partial \omega_n^{\alpha_n}}  \Phi(\omega) \; d\omega_2\right]d\omega_1\ldots d\omega_n \\
& \quad \vdots \\
&  =
(-1)^{\alpha_1+\ldots+\alpha_n}\int_{a_1}^{b^1}\ldots \int_{a_n}^{b^n} \frac{\partial^{\alpha_n}}{\partial \omega_n^{\alpha_n}} \ldots \frac{\partial^{\alpha_1}}{\partial \omega_1^{\alpha_1}}F(\omega) \odot \Phi(\omega) d\omega_n\ldots d\omega_1 \\
&  =
(-1)^{\lvert \alpha \rvert}\int_\Omega D^\alpha F(\omega) \odot \Phi(\omega) d\omega.
\end{align*}
}
            
\end{enumerate}
\end{proof}

\section{Sobolev Spaces of $A$-linearly interactive fuzzy complex process}\label{sec:sobolev}

In classical functional analysis, Sobolev spaces are vector spaces of functions that possess weak derivatives up to a certain order. These spaces play a central role in the study of partial differential equations, since solutions—when they exist—naturally lie in suitable Sobolev spaces.

In this section, we define these spaces for $A$-linearly interactive fuzzy complex processes. Throughout our discussion, we consider $\Omega \subseteq \mathbb{R}^n$ to be an open domain. The proposed spaces will inherit both the analytical structure of classical Sobolev spaces and the algebraic properties of fuzzy complex numbers developed in the preceding sections. 

A weak derivative is a generalization of the concept of the derivative of a (classic) function for functions not assumed differentiable, but only integrable. 

\begin{definition}[$\alpha^{th}$-weak partial derivative] 
Let $L^1_{\text{loc}}(\Omega,\CFA)$ be the set of all functions that are integrable in any compact set $K$ of  $\Omega$. 
For $F,G \in L^1_{\text{loc}}(\Omega,\CFA)$ and a multi-index $\alpha$, we say that $G$ is the $\alpha^{th}$-weak partial derivative of $F$ if, for all $\Phi \in C_{c}^{\infty}(\Omega,\CFA)$, it follows 
\begin{equation}\label{weak}
\int_{\Omega} F(\omega) \odot D^{\alpha} \Phi(\omega) \;d\omega = (-1)^{\mid \alpha \mid} \int_{\Omega} G(\omega) \odot \Phi(\omega) \;d\omega\,. 
\end{equation}
In this case, we write $G = D^{\alpha} F$\footnote{It is common to use the same symbol of the differential operator $D^\alpha$ given by Equation~\eqref{eq:OperatorD_alpha}, since every $F \in C^k(\RR^n,\CFA)$ satisfies Equation~\eqref{eq:parts_generalized}.}.
\end{definition}

Of course, if $F$ is differentiable in the usual sense, then the notion of usual and weak derivatives coincide.

\begin{proposition}\label{prop:weak_unique}
A weak $\alpha^{th}$-partial derivative of $F \in L^1_{\text{loc}}(\Omega,\CFA)$, if it exists, is uniquely determined almost everywhere in $\Omega$.
\end{proposition}
\begin{proof}
Assume $G,\tilde{G} \in L^1_{\text{loc}}(\Omega,\mathbb{C}_{\mathcal{F}(A)})$ both satisfy Equation~\eqref{weak} for the $\alpha^{th}$-weak derivative of $F$. Then for all $\Phi \in C_{c}^{\infty}(\Omega,\CFA)$:
$$\int_{\Omega} (G \ominus \tilde{G})(\omega) \odot \Phi(\omega)\,d\omega = 0.$$
Let $G(\omega) = u(\omega) + v(\omega)A$ and $\tilde{G}(\omega) = \tilde{u}(\omega) + \tilde{v}(\omega)A$ for all $\omega \in \Omega$, with   $u,v,\tilde{u},\tilde{v}:\Omega\to\CC.$
For all $\phi \in C_{c}^{\infty}(\Omega,\CC)$, we have  
$\Phi = \phi + 0A \in C_{c}^{\infty}(\Omega,\CFA)$. From Equation~\eqref{eq:A_product}, we obtain 
\begin{eqnarray*}
0 & = & \int_{\Omega} (G \ominus \tilde{G})(\omega) \odot \Phi(\omega)\,d\omega \\
& = & \int_{\Omega} \left[(u(\omega)-\tilde{u}(\omega)) + (v(\omega)-\tilde{v}(\omega))A\right] \odot \left[\phi(\omega) + 0A\right]\,d\omega \\
& = & \int_{\Omega} \phi(\omega)(u(\omega)-\tilde{u}(\omega))\,d\omega  + \left(\int_{\Omega} \phi(\omega)(v(\omega)-\tilde{v}(\omega))\,d\omega \right) A \\
\Rightarrow & & \int_{\Omega} \phi(\omega)(u(\omega)-\tilde{u}(\omega))\,d\omega  = \int_{\Omega} \phi(\omega)(v(\omega)-\tilde{v}(\omega))\,d\omega = 0
\end{eqnarray*}
for all $\phi \in C_{c}^{\infty}(\Omega,\CC).$
Since $\phi$ is arbitrary, the Fundamental Lemma of Calculus of Variations implies $u = \tilde{u}$ and $v = \tilde{v}$ almost everywhere. Therefore, 
$G = \tilde{G}$ almost everywhere.
\end{proof}

\begin{example}\label{ex:weak_deriv}
Consider $F(t) = q_1(t)A$ and $G(t) = q_2(t)A$ where 
\[
q_1(t) = \begin{cases}
te^{ip} & \text{if } -1 < t \leq 0 \\
e^{ip} & \text{if } 0 < t < 1
\end{cases}, \quad 
q_2(t) = \begin{cases}
e^{ip} & \text{if } -1 < t \leq 0 \\
0 & \text{if } 0 < t < 1
\end{cases}
\]
with $p \in [0,\frac{\pi}{2}]$. Let us verify that $G$ is the weak derivative of $F$ in the interval $(-1,0)$. For any test function $\Phi \in C_c^\infty([-1,0],\CFA)$ we have $\Phi(-1)=\Phi(0)=0$, and the weak derivative condition is:
\begin{equation}
\int_{-1}^0 F(\omega) \odot \Phi'(t) d\omega = - \int_{-1}^0 G(\omega) \odot \Phi(\omega) d\omega.
\end{equation}
By integration by parts (Proposition~\ref{partes}(d)), we have:
\begin{align*}
\int_{-1}^0 F(t) \odot \Phi'(t) dt 
&= F(t) \odot \Phi(t)\Big\vert_{-1}^0 \ominus \int_{-1}^0 F'(t) \odot \Phi(t) dt \\
&= - \int_{-1}^0 e^{ip}A \odot \Phi(t) dt \\
&= - \int_{-1}^0 G(t) \odot \Phi(t) dt
\end{align*}
Thus $F' = G$ in the weak sense in $(-1,0)$.
\end{example}

\begin{definition}\label{def:Lp_space}
Let $1 \leq p < \infty$. We define the space 
$L^p(\Omega,\CFA)$ given by 
\begin{equation}
L^p(\Omega,\CFA) = \left\lbrace F \in \FP(\Omega,\CFA) \mid \int_\Omega \|F(\omega)\|_{\tilde{p}}^p \;d\omega < \infty \right\rbrace.    
\end{equation}
Furthermore, consider the function $\|\cdot\|_{\tilde{p}}:L^p(\Omega,\CFA) \to [0,\infty)$ given by 
\begin{equation}
\|F\|_{\tilde{p}} = \left(\int_\Omega \|F(\omega)\|_{\tilde{p}}^p\; d\omega\right)^{\frac{1}{p}} = \left(\int_\Omega 
\lvert u(\omega) + av(\omega)\rvert^p + \lvert v(\omega)\rvert^p
\; d\omega\right)^{\frac{1}{p}},  
\end{equation}
where $F(\omega) = u(\omega) + v(\omega)A$, for all $\omega \in \Omega$ with $u,v:\Omega\to \CC$, and $a$ is the midpoint of the 1-level set of $A$.
\end{definition}

The following theorem associates the space $L^p(\Omega,\CFA)$ with the vector space forms by Cartesian product of the classical complex Banach space $L^p(\Omega,\CC)$. Moreover, it also a suitable equivalence of norms on $L^p(\Omega,\CFA).$  

\begin{theorem}\label{thm:charaterize_Lp}
Let $1 \leq p < \infty$. 
The space $L^p(\Omega,\CFA)$ is isomorphic to the vector space 
$L^p(\Omega,\CC)\times L^p(\Omega,\CC)$ via the isomorphism $(u + vA) \mapsto (u,v)$.

Furthermore, the norm $\|\cdot\|_{\tilde{p}}$ on $L^p(\Omega,\CFA)$ is equivalent to the norm $\|\cdot\|_{p,\infty}$ on $L^p(\Omega,\CFA)$ defined by 
\begin{equation}
\| u + vA\|_{m-p} = \max\{\|u\|_p,\|v\|_p\}    
\end{equation}
satisfying 
\[
\frac{1}{\left\|M^{-1}\right\|_p} \| F\|_{m-p}  \leq  \|F\|_{\tilde{p}} \leq 2^{\frac{1}{p}}\|M\|_p \| F\|_{m-p}\,,
\]
where $M$ is the non-singular matrix given by Equation~\eqref{eq:M}.
\end{theorem}
\begin{proof}
Given $u+vA \in L^p(\Omega,\CFA)$ with $u,v:\Omega\to \CC.$ 
This implies that
\begin{eqnarray*}
\|F\|_{\tilde{p}}^p & = & 
\int_\Omega \|F(\omega)\|_{\tilde{p}}^p \;d\omega 
 < \infty \\
 & = & \int_\Omega \lvert u(\omega) + av(\omega)\rvert^p + \lvert v(\omega)\rvert^p
\; d\omega  
 < \infty \\
\Rightarrow & & \int_\Omega \lvert u(\omega) + av(\omega)\rvert^p 
\; d\omega  < \infty \;\mbox{ and }\; 
\int_\Omega \lvert v(\omega)\rvert^p
\; d\omega  
 < \infty \\
\Rightarrow & & (u + av), v \in L^p(\Omega,\CC).
\end{eqnarray*}
Since $L^p(\Omega,\CC)$ is a vector space, we have $u = ((u + av) - av) \in L^p(\Omega,\CC)$. On the other hand, if $u,v \in L^p(\Omega,\CC)$, then 
$(u + av) \in L^p(\Omega,\CC)$. This implies that
\[
\int_\Omega \lvert u(\omega) + av(\omega)\rvert^p 
\; d\omega  < \infty \;\mbox{ and }\; 
\int_\Omega \lvert v(\omega)\rvert^p
\; d\omega  
 < \infty \Rightarrow (u + vA) \in L^p(\Omega,\CFA)\,.
\]
Hence, the mapping $(u+vA)\mapsto (u,v)$ is indeed a well-defined bijection from $L^p(\Omega,\CFA)$ to $L^p(\Omega,\CC)\times L^p(\Omega,\CC).$ Obviously, this bijective mapping is a linear operator. Therefore, we conclude that the space $L^p(\Omega,\CFA)$ is isomorphic to the vector space 
$L^p(\Omega,\CC)\times L^p(\Omega,\CC)$. 

Let $F = (u+vA) \in L^p(\Omega,\CFA)$ and $M$ be the matrix given by Equation~\eqref{eq:M}. 
Using the consistency of the $p$-norm of matrix and vector, 
for $\omega \in \Omega$ we obtain  
\[
\|F(\omega)\|_{\tilde{p}} = \left\|M\begin{bmatrix}
u(\omega) \\ v(\omega)    
\end{bmatrix} \right\|_p \leq \|M\|_p \left\|\begin{bmatrix}
u(\omega) \\ v(\omega)    
\end{bmatrix} \right\|_p  =  \|M\|_p\left(\lvert u(\omega) \rvert^p + \lvert v(\omega) \rvert^p \right)^{\frac{1}{p}}\,,
\]
where $\|M\|_p$ stands for the $p$-norm of the matrix $M$. Moreover, we have 
\[
\left(\lvert u(\omega) \rvert^p + \lvert v(\omega) \rvert^p \right)^{\frac{1}{p}} = \left\|M^{-1}M\begin{bmatrix}
u(\omega) \\ v(\omega)    
\end{bmatrix} \right\|_p \leq \left\|M^{-1}\right\|_p  \left\|M\begin{bmatrix}
u(\omega) \\ v(\omega)    
\end{bmatrix} \right\|_p = \left\|M^{-1}\right\|_p
\|F(\omega)\|_{\tilde{p}},
\]
which implies that 
\[
\frac{1}{\left\|M^{-1}\right\|_p} \left(\lvert u(\omega) \rvert^p + \lvert v(\omega) \rvert^p \right)^{\frac{1}{p}} \leq \|F(\omega)\|_{\tilde{p}}.
\]
Thus, we have 
\begin{align*}
& \frac{1}{\left\|M^{-1}\right\|_p^p} \left(\lvert u(\omega) \rvert^p + \lvert v(\omega) \rvert^p \right) \leq  \|F(\omega)\|_{\tilde{p}}^p  \leq   \|M\|_p^p\left(\lvert u(\omega) \rvert^p + \lvert v(\omega) \rvert^p \right) \\
\Rightarrow & \frac{1}{\left\|M^{-1}\right\|_p} \left(\int_\Omega  \lvert u(\omega) \rvert^p + \lvert v(\omega) \rvert^p d\omega \right)^{\frac{1}{p}} \leq  \left(\int_\Omega  \|F(\omega)\|_{\tilde{p}}^p d\omega \right)^{\frac{1}{p}} \leq \\ & \qquad \leq   \|M\|_p \left(\int_\Omega \lvert u(\omega) \rvert^p + \lvert v(\omega) \rvert^p d\omega \right)^{\frac{1}{p}} 
\end{align*}
In addition, we have 
\[
\left\lbrace\begin{array}{l}
\|u\|_p^p =  \int_\Omega  \lvert u(\omega) \rvert^p d\omega \leq \int_\Omega  \lvert u(\omega) \rvert^p + \lvert v(\omega) \rvert^p d\omega \leq 2\max\{\|u\|_p^p,\|v\|_p^p\}  \\
\|v\|_p^p =  \int_\Omega  \lvert v(\omega) \rvert^p d\omega \leq \int_\Omega  \lvert u(\omega) \rvert^p + \lvert v(\omega) \rvert^p d\omega \leq 2\max\{\|u\|_p^p,\|v\|_p^p\}
\end{array}\right.
\]
which implies that
\[
\max\{\|u\|_p,\|v\|_p\} \leq \left(\int_\Omega \lvert u(\omega) \rvert^p + \lvert v(\omega) \rvert^p d\omega \right)^{\frac{1}{p}} \leq  2^{\frac{1}{p}}\max\{\|u\|_p,\|v\|_p\}\,.
\]
Combining the inequalities above, we conclude that
\begin{align*}
& \frac{1}{\left\|M^{-1}\right\|_p} \max\{\|u\|_p,\|v\|_p\} \leq  \|F\|_{\tilde{p}} \leq 2^{\frac{1}{p}}\|M\|_p \max\{\|u\|_p,\|v\|_p\}\\
\Rightarrow & 
\frac{1}{\left\|M^{-1}\right\|_p} \| F\|_{m-p}  \leq  \|F\|_{\tilde{p}} \leq 2^{\frac{1}{p}}\|M\|_p \| F\|_{m-p} \,.    
\end{align*}
\end{proof}

\begin{example}\label{ex:L1_function}
Let $\Omega = (0,2\pi)$ and $A \in \RF$ be asymmetric. Consider the function $F(x) = \left(\frac{1}{e^2}ie^{ix}\right)A$. 
From Theorem~\ref{thm:charaterize_Lp}, we have that  $F \in L^1(\Omega,\CFA)$ since $q(x) = \frac{1}{e^2}ie^{ix} \in L^1(\Omega,\CC)$, $r(x) = 0 \in L^1(\Omega,\CC)$. 
\end{example}

\begin{corollary}\label{cor:Lp_banach}
$L^p(\Omega,\mathbb{C}_{\mathcal{F}(A)})$ is a Banach space for $1 \leq p < \infty$.
\end{corollary}

\begin{definition}[Fuzzy Sobolev Space]
For $1 \leq p < \infty$ and $k \in \mathbb{N}$, the fuzzy Sobolev space $W^{k,p}(\Omega,\CFA)$ consists of all functions $F \in L^{p}(\Omega,\CFA)$ possessing weak derivatives $D^{\alpha} F \in L^{p}(\Omega,\mathbb{C}_{\mathcal{F}(A)})$ for all multi-index $\alpha$ with $\mid\alpha\mid\leq k$. 
When $p=2$, we simply denote $W^{k,2}(\Omega, \CFA)$ by $H^k(\Omega, \CFA)$ for all integer $k \geq 0$. Note that $H^0(\Omega, \CFA)=L^2(\Omega, \CFA)$. 

For $p \geq 1$, we define the norm $\|\cdot\|_{\tilde{p}}$ on $W^{k,p}(\Omega,\CFA)$ given, for every $F \in W^{k,p}(\Omega,\CFA)$, by
\begin{equation}
\|F\|_{\tilde{p},k} = \sum_{\mid\alpha\mid \leq k} \|D^{\alpha}F\|_{\tilde{p}}.
\end{equation} 
\end{definition}

The following theorem states some properties of the fuzzy Sobolev space $W^{k,p}(\Omega,\CFA)$.  

\begin{theorem}[Sobolev Space Properties]\label{thm:sobolev_props}
Assume $F,G \in W^{k,p}(\Omega,\CFA)$, for $k \in \mathbb{N}$, $\mid \alpha \mid \leq k$, $1 \leq p \leq \infty$. The following properties hold.
\begin{enumerate}
\item[i)] Regularity reduction: For any multi-indices $\alpha,\beta$ with $\mid \alpha \mid +\mid \beta \mid \leq k$, the commutation property is satisfied:
\begin{equation}
D^{\beta}(D^{\alpha} F) = D^{\alpha}(D^{\beta} F) = D^{\alpha+\beta}F
\end{equation}
\item[ii)] Linearity: For all $c_1,c_2 \in \CC$ :
\begin{equation}
D^{\alpha}(c_1 F \oplus c_2  G) = c_1  D^{\alpha}F \oplus c_2  D^{\alpha}G
\end{equation}
\end{enumerate}
\end{theorem}
\begin{proof}
To prove item i), consider an arbitrary $\Phi \in C^{\infty}_{c}(\Omega,\CFA)$. 
For the fuzzy Clairaut-Schwarz theorem, we have $D^{\alpha}(D^{\beta}\Phi) = D^{\beta}(D^{\alpha}\Phi) = D^{\alpha+\beta} \Phi$. Moreover,  $D^{\beta}\Phi, D^{\alpha}\Phi \in C^{\infty}_{c}(\Omega,\CFA)$. 
Note that $F,G \in W^{k,p}(\Omega,\CFA)$ implies that $D^\alpha F \in W^{k,p-\lvert\beta\rvert}(\Omega,\CFA)$ and $D^\beta F\in W^{k,p-\lvert\alpha\rvert}(\Omega,\CFA)$. This implies that 
\begin{align*}
(-1)^{\lvert \alpha + \beta \rvert}\int_{\Omega} D^{\alpha+\beta}F(\omega) \odot \Phi(\omega) d\omega & = 
\int_{\Omega} F(\omega) \odot D^{\alpha+\beta}\Phi(\omega) d\omega \\ & = 
\int_{\Omega} F(\omega) \odot D^{\alpha}(D^{\beta}\Phi)(\omega) d\omega \\ 
& = (-1)^{\lvert \alpha \rvert} \int_{\Omega} D^{\alpha}F(\omega) \odot D^{\beta}\Phi(\omega) d\omega \\
& = (-1)^{\lvert \alpha \rvert}(-1)^{\lvert \beta \rvert} \int_{\Omega} D^{\beta}D^{\alpha}F(\omega) \odot \Phi(\omega) d\omega \\
& = (-1)^{\lvert \alpha +\beta \rvert} \int_{\Omega} D^{\beta}(D^{\alpha}F)(\omega) \odot \Phi(\omega) d\omega
\end{align*}
for all $\Phi \in C^{\infty}_{c}(\Omega,\CFA)$. 
From Proposition~\ref{prop:weak_unique}, we conclude that 
$D^{\beta}(D^{\alpha}F) = D^{\alpha + \beta}F$ almost everywhere. 
Similar, one can prove that $D^{\alpha}(D^{\beta}F) = D^{\alpha + \beta}F$ almost everywhere.

To prove (ii), we can use the distributive property of $\odot$ and the linear property of the integral:
\begin{align*}
& \int_{\Omega} (c_1F \oplus c_2G)(\omega) \odot D^{\alpha}\Phi(\omega) d\omega = \\
& \quad = c_1 \int_{\Omega} F(\omega) \odot D^{\alpha}\Phi(\omega)d\omega \oplus c_2 \int_{\Omega} G(\omega) \odot D^{\alpha}\Phi(\omega) d\omega \\
&\quad  = (-1)^{\mid\alpha\mid} \left[ c_1 \int_{\Omega} D^{\alpha}F(\omega) \odot \Phi(\omega) d\omega \oplus c_2 \int_{\Omega} D^{\alpha}G(\omega) \odot \Phi(\omega)d\omega \right] \\
&\quad = (-1)^{\mid\alpha\mid} \int_{\Omega} (c_1D^{\alpha}F \oplus c_2D^{\alpha}G)(\omega) \odot \Phi(\omega) d\omega
\end{align*}
for all $\Phi \in C^{\infty}_{c}(\Omega,\CFA).$
\end{proof}

\begin{theorem}\label{thm:characterization_Sobolev}
Consider the spaces $(W^{k,p}(\Omega,\CFA), \|\cdot\|_{\tilde{p},k})$ and  
$\bigl(W^{k,p}(\Omega,\CC)\times W^{k,p}(\Omega,\CC),\, \|\cdot\|_{m-p,k}\bigr)$,
where the norm $\|\cdot\|_{m-p,k}$ is defined by
\[
\|(u,v)\|_{m-p,k}
    := \sum_{|\alpha|\le k} \max\{\|D^\alpha u\|_{p},\, \|D^\alpha v\|_{p}\}.
\]

Le $\Psi : W^{k,p}(\Omega,\CFA) \longrightarrow  W^{k,p}(\Omega,\CC)\times W^{k,p}(\Omega,\CC)$ be the mapping given by 
\[
F = u + vA \longmapsto \Psi(F) = (u,v).
\]
The operator $\Psi$ is a linear homeomorphism, that is, $\Psi$ is bijective and both $\Psi$ and $\Psi^{-1}$ are continuous.  
Consequently, the fuzzy Sobolev space $(W^{k,p}(\Omega,\CFA),\|\cdot\|_{p,k})$ is isomorphic as a vector space and homeomorphic as a normed space to $W^{k,p}(\Omega,\CC)\times W^{k,p}(\Omega,\CC)$.
\end{theorem}
\begin{proof}
On the one hand, the space $W^{k,p}(\Omega,\CFA)$ is clearly a vector space, since $D^\alpha$ is a linear operator by Proposition~\ref{thm:sobolev_props}. It is also evident that the function $\|\cdot\|_{\tilde{p},k}$ is a norm on $W^{k,p}(\Omega,\CFA)$, since $\|\cdot\|_{\tilde{p}}$ is a norm on $L^p(\Omega,\CFA)$. On the other hand, $\bigl(W^{k,p}(\Omega,\CC)\times W^{k,p}(\Omega,\CC),\, \|\cdot\|_{m-p,k}\bigr)$ is a Banach space because it corresponds to the Cartesian product of two classical Sobolev spaces. Note that one can easily verify that $\|\cdot\|_{m-p,k}$ is indeed a norm on 
$W^{k,p}(\Omega,\CC)\times W^{k,p}(\Omega,\CC)$. 

Since the vector addition and scalar multiplication in $W^{k,p}(\Omega,\CFA)$ and $W^{k,p}(\Omega,\CC)\times W^{k,p}(\Omega,\CC)$ are defined pointwise, it follows that $\Psi$ is a linear operator. Given $F,G \in W^{k,p}(\Omega,\CC)$ with $F = u + vA$ and $G = z + wA$, $u,v,z,w:\Omega\to \CC$. From Proposition~\ref{prop:Basic_calculus_property}(c) and Theorem~\ref{thm:charaterize_Lp}, the fact $D^\alpha F,D^\alpha G \in L^p(\Omega,\CFA)$ implies that $D^\alpha F = (D^\alpha u) + (D^\alpha v) A$, $D^\alpha G = (D^\alpha z) + (D^\alpha w) A$, and  $D^\alpha u, D^\alpha v, D^\alpha z, D^\alpha w \in L^p(\Omega,\CC)$ for all multi-index $\alpha$ satisfying $\lvert \alpha \rvert \leq k$. Conversely, if $u,v,z,w \in L^p(\Omega,\CC)$, then $D^\alpha u, D^\alpha v, D^\alpha z, D^\alpha w \in L^p(\Omega,\CC)$ for all multi-index $\alpha$ satisfying $\lvert \alpha \rvert \leq k$. By applying Theorem~\ref{thm:charaterize_Lp} and Proposition~\ref{prop:Basic_calculus_property}, we obtain 
$D^\alpha F = (D^\alpha u) + (D^\alpha v) A \in  L^p(\Omega,\CFA)$, $D^\alpha G = (D^\alpha z) + (D^\alpha w) A  \in  L^p(\Omega,\CFA)$ for all $\lvert \alpha \rvert \leq k$. Therefore, we conclude that $\Psi$ is a bijection. 

By Theorem~\ref{thm:charaterize_Lp}, for every multi-index $\alpha$, $\lvert \alpha \rvert \leq k$ and $F=(u+vA) \in W^{k,p}(\Omega,\CFA)$, we have 
\[
\frac{1}{\left\|M^{-1}\right\|_p} \| D^\alpha F\|_{m-p}  \leq  \|D^\alpha F\|_{\tilde{p}} \leq 2^{\frac{1}{p}}\|M\|_p \| D^\alpha F\|_{m-p}\,.
\]
This implies that 
\[
\frac{1}{\left\|M^{-1}\right\|_p} \left(\underbrace{\sum_{\lvert \alpha\rvert \leq k} \| D^\alpha F\|_{m-p}}_{= \|(u,v)\|_{m-p,k}} \right) \leq  \underbrace{\sum_{\lvert \alpha\rvert \leq k}\|D^\alpha F\|_{\tilde{p}}}_{= \|F\|_{\tilde{p},k}} \leq 2^{\frac{1}{p}}\|M\|_p \left(\underbrace{\sum_{\lvert \alpha\rvert \leq k}\| D^\alpha F\|_{m-p}}_{= \|(u,v)\|_{m-p,k}} \right)
\]
\[
\Rightarrow \frac{1}{\left\|M^{-1}\right\|_p} \|(u,v)\|_{m-p,k} \leq \|F\|_{\tilde{p},k} \leq 
2^{\frac{1}{p}}\|M\|_p \|(u,v)\|_{m-p,k}\,.
\]
The last inequalities ensure that both $\Psi$ and $\Psi^{-1}$ are continuous. 
\end{proof}

As an immediate consequence of Theorem~\ref{thm:characterization_Sobolev} is that the fuzzy Sobolev space is complete, since the space $W^{k,p}(\Omega,\CC)\times W^{k,p}(\Omega,\CC)$ is a Banach space.  

\begin{corollary}[Completeness of Fuzzy Sobolev Spaces]\label{cor:sobolev_banach}
Let $k \in \mathbb{N}$ and $1 \leq p < \infty$. The fuzzy Sobolev space $(W^{k,p}(\Omega,\CFA),\|\cdot\|_{\tilde{p},k})$ is a Banach space.
\end{corollary}

The novelty here lies in the consistent use of the $A$-cross product ($\odot$) and Sobolev norms $\|\cdot\|_{\tilde{p},k}$ to maintain algebraic closure while preserving the analytic structure of solutions.

\section{Fourier transform in $\CFA$}\label{sec:Fourier}

\begin{definition}[Fourier transform in $L^{1}(\RR^n,\CFA)$]\label{DefinitionTFourier}
The Fourier transform of a function $F\in L^{1}(\RR^n,\CFA)$ is the function  $\widehat{F}:\RR^{n}\to\CFA$ defined by
\begin{equation}\label{eq:Fourier}
\widehat{F}(\omega) = \int_{\RR^n}{F(\xi)e^{-i\langle \omega,\xi\rangle}d\xi}\,,
\end{equation}
where $\langle \omega,\xi\rangle$ is the Euclidean inner product of $\omega \in \RR^n$ and $\xi \in \RR^n$. 
\end{definition}

For $F \in L^{1}(\RR^n,\CFA)$, there exist $u,v \in L^{1}(\RR^n,\CC)$ such that $F = u + vA.$ Thus, it follows that 
\begin{align}\label{eq:Fourier_coef} \nonumber
\widehat{F}(\omega) & =\int_{\RR^n}{F(\xi)e^{-i\langle \omega,\xi\rangle}d\xi} \\ \nonumber
 & =\underbrace{\int_{\RR^n}u(\xi)e^{-i\langle \omega,\xi\rangle} d\xi}_{=\hat{u}(\xi)} + \underbrace{\left(\int_{\RR^n}v(\xi)e^{-i\langle \omega,\xi\rangle} d\xi\right)}_{=\hat{v}(\xi)}A \\
 & = \hat{u}(\xi) + \hat{v}(\xi)A\,,
\end{align}
where $\hat{u}(\xi)$ and $\hat{v}(\xi)$ correspond to the classical Fourier transform of $u$ and $v$, respectively. 

The connections between the fuzzy Lebesgue space and the product of classical Lebesgue spaces, as well as between the fuzzy Fourier transform and the classical Fourier transform, allow us to transfer many well-known results from the crisp setting to the fuzzy one. 

\begin{proposition}\label{prop:Prop_Fourier}
Let $F,G \in L^1(\RR^n,\CFA)$, $a \in \CC$, $B\in \CFA$, and $x \in \RR^n$.
\begin{itemize}
    \item[1.] Fuzzy Riemann-Lebesgue Lemma: $\widehat{F} \in C^0(\RR^n,\CFA)$, $\|\widehat{F}\|_{\tilde{\infty}} \leq \|\widehat{F}\|_{\tilde{1}}$, and $\widehat{F}(\omega) \to 0$ when $\lvert \omega \rvert \to \infty$; 
    \item[2.] $\left(\widehat{F\left(\frac{\xi}{a}\right)}\right)(\omega) =  a^n \widehat{F}(a\omega)$;
    \item[3.] $\left(\widehat{e^{i\langle x, \xi\rangle} F\left(\xi\right)}\right)(\omega) =  \widehat{F}(\omega - x)$;
    \item[4.] $\left(\widehat{F\left(\xi -x\right)}\right)(\omega) =  e^{-i\langle x, \omega\rangle} \widehat{F}(\omega)$;
    \item[5.] If $F \in C^1(\RR^n,\CFA)$ and $\frac{\partial}{\partial \xi_i}F \in L^1(\RR^n,\CFA)$, $i\in \{1,\ldots,n\}$, then 
    \begin{equation}
        \widehat{\frac{\partial}{\partial \xi_i}F}(\omega) = i\omega_i\widehat{F}(\omega)
    \end{equation}
    \item[6.] Inverse fuzzy Fourier transform: If $\hat{F} \in L^1(\RR^n,\CFA)$, then 
    \begin{equation}
        F(\xi) = \frac{1}{(2\pi)^n}\int_{\RR^n}\widehat{F}(\omega)e^{i\langle \omega, \xi\rangle} \,d\omega \qquad \mbox{ a.e.}
    \end{equation}    
    \item[7.] Linearity: $\widehat{F\oplus B\odot G}(\omega) = \widehat{F}(\omega)\oplus B\odot\widehat{G}(\omega)$;
    \item[8.] $\left(\widehat{F\star G}\right)(\omega) =  \widehat{F}(\omega)\odot \widehat{G}(\omega)$ with $(F\star G)(\xi) = \int_{\RR^n}F(z)\odot G(\xi - z)\,dz$. 
    \item[9.] If $F \in L^2(\RR^n,\CFA)$, then $\widehat{F} \in L^2(\RR^n,\CFA)$ and 
    \[
    \|\widehat{F}\|_{\tilde{2}} \leq (2\pi)^{\frac{n}{2}}\sqrt{2} \|{F}\|_{\tilde{2}}\,. 
    \]    
\end{itemize}
\end{proposition}
\begin{proof}
First, note that these 8 properties are satisfied for any $u,v \in L^1(\RR^n,\CC)$  \cite{Linares}.  Thus, using this fact together with Equation~\eqref{eq:Fourier_coef}, we obtain that the proof for the properties 1.--6. follow immediately. Property 7. is obtained by applying Proposition~\ref{prop:Basic_calculus_property}(g) to Equation~\eqref{eq:Fourier}.

Let $F=u + vA$ and $G=z + wA$ for some $u,v,z,w \in L^1(\RR^n,\CC)$. 
From Propositions~\ref{prop:inequality_norm_1} and \ref{prop:Basic_calculus_property} and from Fubini-Tonelli theorem, we have 
\begin{align*}
\|F\star G\|_{\tilde{1}} & \leq \int_{\RR^n}\int_{\RR^n} \|F(z)\odot G(\xi - z)\|_{\tilde{1}}\,dz\,dx \\
 & = \int_{\RR^n}\|F(z)\|_{\tilde{1}}\int_{\RR^n} \|G(\xi - z)\|_{\tilde{1}}\,d\xi\,dz \\
 & = \int_{\RR^n}\|F(z)\|_{\tilde{1}}\int_{\RR^n} \|G(\xi\|_{\tilde{1}}\,d\xi\,dz \\
 & = \|G\|_{\tilde{1}}\int_{\RR^n}\|F(z)\|_{\tilde{1}}\,dz \\
& = \|G\|_{\tilde{1}}\|F\|_{\tilde{1}}.
\end{align*}
This implies that $F\star G \in L^1(\RR^n,\CFA)$, and therefore its fuzzy Fourier transform is defined. 
Let $a = m(A)$. We obtain

{\small
\begin{align*}
\widehat{F\star G}(\omega) & = \int_{\RR^n}{(F\star G)(\xi)e^{-i\langle \omega,\xi\rangle}d\xi} \\ 
& = \int_{\RR^n}{\left(\int_{\RR^n}F(y)\odot G(\xi - y)dy\right) e^{-i\langle \omega,\xi\rangle}\, d\xi} \\ 
& = \int_{\RR^n}\int_{\RR^n}(u(y)z(\xi -y) -a^2v(y)w(\xi -y))e^{-i\langle \omega,\xi\rangle} \, d\xi dy + \\
& \quad +  A\int_{\RR^n}\int_{\RR^n}(u(y)w(\xi - y) + v(y)z(\xi - y) + 2av(y)w(\xi - y))e^{-i\langle \omega,\xi\rangle}\, d\xi dy \\ 
& = \int_{\RR^n}u(y)\int_{\RR^n}z(\xi -y)e^{-i\langle \omega,\xi\rangle}\, d\xi dy -a^2\int_{\RR^n}v(y)\int_{\RR^n}w(\xi -y)e^{-i\langle \omega,\xi\rangle}\, d\xi dy + \\
& \quad +  A\int_{\RR^n}u(y)\int_{\RR^n}w(\xi - y)e^{-i\langle \omega,\xi\rangle}\, d\xi dy + A\int_{\RR^n}v(y)\int_{\RR^n}z(\xi - y)e^{-i\langle \omega,\xi\rangle}\, d\xi dy \\ 
& \quad + 2aA\int_{\RR^n}v(y)\int_{\RR^n}w(\xi - y)e^{-i\langle \omega,\xi\rangle}\, d\xi dy \\ 
& = \int_{\RR^n}u(y) \widehat{z}(\omega)e^{-i\langle y,\omega\rangle}\, dy -a^2\int_{\RR^n}v(y)\widehat{w}(\omega)e^{-i\langle y,\omega\rangle}\, dy + \\
& \quad +  A\int_{\RR^n}u(y)\widehat{w}(\omega)e^{-i\langle y,\omega\rangle}\, dy + A\int_{\RR^n}v(y)\widehat{z}(\omega)e^{-i\langle y,\omega\rangle}\, dy \\ 
& \quad + 2aA\int_{\RR^n}v(y)\widehat{w}(\omega)e^{-i\langle y,\omega\rangle}\, dy \\ 
& = \widehat{u}(\omega)\widehat{z}(\omega)-a^2\widehat{v}(\omega)\widehat{w}(\omega) +   \left(\widehat{u}(\omega)\widehat{w} + \widehat{v}(\omega)\widehat{z}(\omega)+ 2a\widehat{v}(\omega)\widehat{w}(\omega)\right)A \\ 
& = \widehat{F}(\omega)\odot \widehat{G}(\omega)\,.
\end{align*}
}
This concludes the proof of Property 8.

Let $F = u + vA \in L^1(\RR^n,\CFA) \cap L^2(\RR^n,\CFA)$. This implies that $u,v \in L^1(\RR^n,\CC) \cap L^2(\RR^n,\CC)$ and, from Plancherel's theorem \cite{Linares}, we have $\widehat{u},\widehat{v} \in L^2(\RR^n,\CC)$ with 
\[
\|\widehat{u}\|_2 \leq (2\pi)^{\frac{n}{2}}\|{u}\|_2 \;\mbox{ and }\; \|\widehat{v}\|_2 \leq (2\pi)^{\frac{n}{2}}\|{v}\|_2\,.
\]
From Theorem~\ref{thm:charaterize_Lp}, we obtain 
\begin{align*}
\|\widehat{F}\|_{\tilde{2}} & \leq 2^{\frac{1}{2}}\|M\|_2 \max\{\|\widehat{u}\|_2,\|\widehat{v}\|_2\} \\
&\leq (2\pi)^{\frac{n}{2}}2^{\frac{1}{2}}\|M\|_2 \max\{\|{u}\|_2,\|{v}\|_2\} \\
&\leq (2\pi)^{\frac{n}{2}}2^{\frac{1}{2}}\|M\|_2\|M^{-1}\|_2 \|{F}\|_{\tilde{2}}\,,
\end{align*}
where $M$ is the matrix given by Equation~\eqref{eq:M}. Note that $\|M\|_2 = \|M^{-1}\|_2 =1$, since $1$ is the unique eigenvalue of $M$ and $M^{-1}$. Thus, we conclude that 
\[
\|\widehat{F}\|_{\tilde{2}} \leq (2\pi)^{\frac{n}{2}}\sqrt{2}\|{F}\|_{\tilde{2}}\,.
\]
\end{proof}

Property 9. of Proposition~\ref{prop:Prop_Fourier} ensures that 
the fuzzy Fourier transform regarded as an operator defined from   $L^1(\RR^n,\CFA)\cap L^2(\RR^n,\CFA)$ to $L^2(\RR^n,\CFA)$ is uniformly continuous. Since $L^1(\RR^n,\CC)\cap L^2(\RR^n,\CC)$ is dense in $L^2(\RR^n,\CC)$, from Theorem~\ref{thm:charaterize_Lp}, it follows that 
$L^1(\RR^n,\CFA)\cap L^2(\RR^n,\CFA)$ is dense in $L^2(\RR^n,\CFA).$ 
Since $L^2(\RR^n,\CFA)$ is a Banach space, the fuzzy Fourier transform can be uniformly extended to the entire domain $L^2(\RR^n,\CFA)$ as follows. 

\begin{definition}[Fourier transform in $L^{2}(\RR^n,\CFA)$]
For every $F \in L^{2}(\RR^n,\CFA)$, the fuzzy Fourier(-Plancherel) transform of $F$ is the function $\widehat{F} \in L^{2}(\RR^n,\CFA)$ given by 
\begin{equation}
\lim_{n\to \infty} \|\widehat{F_n} - \widehat{F}\|_{\tilde{2}} = 0   
\end{equation}
for any sequence $(F_n)$ in $L^1(\RR^n,\CFA)\cap L^2(\RR^n,\CFA)$ that converges to $F$ in the norm $\|\cdot\|_{\tilde{2}}.$
\end{definition}

Note that the definition of fuzzy Fourier transform in $L^2(\RR^n,\CFA)$, also called fuzzy Fourier-Plancherel transform, is similar to the definition of the classical Fourier(-Plancherel) transform in $L^2(\RR^n,\CC)$ \cite{Linares},  which is obtained by a uniform extension. 
Again, by Theorem~\ref{thm:charaterize_Lp} and the properties of the fuzzy integral, the fuzzy Fourier-Plancherel transform of a function $F = u+ vA$, with $u,v \in L^2(\RR^n,\CFA)$, is also given by Equation~\eqref{eq:Fourier_coef}; however, in this case, $\widehat{u}$ and $\widehat{v}$ denote the classical Fourier–Plancherel transforms of $u$ and $v$, respectively. This allows us to establish fuzzy counterparts of some well-known results for the Fourier transform on $L^2(\RR^n,\CC)$. 

\begin{theorem}\label{thm:prop_Fourier_L2}
Consider the Fourier-Plancherel transform defined in $L^2(\RR^n,\CFA)$ and $F \in L^2(\RR^n,\CFA)$. 
\begin{itemize}
\item[1.] The Fourier-Plancherel transform of $F$ is given by 
\begin{equation}\label{eq:fourier_L2}
\widehat{F}(\omega) = \lim_{q\to\infty} \displaystyle \int_{\lvert\xi\rvert < q}   F(\xi)e^{-i\langle \omega,\xi\rangle}\,d\xi 
\end{equation}
where $\lvert \xi \rvert$ denotes the Euclidean norm of $\xi \in \RR^n.$
\item[2.] Inverse of the Fourier-Plancherel transform: 
\[
F(\xi) = (2\pi)^{-n}\widehat{\widehat{F}}(-\xi)
\]
for any $\xi \in \RR^n.$ 
\item[3.] If $F\in C^1(\RR^n,\CFA)$ and $\frac{\partial}{\partial \xi_i}F \in L^2(\RR^n,\CFA)$, $i\in \{1,\ldots,n\}$, then 
    \begin{equation}
        \widehat{\frac{\partial}{\partial \xi_i}F}(\omega) = i\xi_i\widehat{F}(\omega)
    \end{equation}
\item[4.] Characterization of the Sobolev space $H^k(\Omega, \CFA)$,  $k \in \mathbb{N} \cup \{0\}$: 
\begin{equation}
F \in H^k(\RR^n, \CFA) \Longleftrightarrow (1 + \lvert\xi \rvert ^2)^k\widehat{F}(\xi) \in L^2(\RR^n,\CFA)\,.
\end{equation}
\end{itemize}
\end{theorem}
\begin{proof}
These results follow from the fact that $\widehat{F}(\omega) = \widehat{u}(\omega)  + \widehat{v}(\omega)A$, with $\widehat{u}, \widehat{v} \in L^2(\RR^n,\CC)$, together with the validity of these results for functions in $L^2(\RR^n,\CC)$ \cite{Linares}.    
\end{proof}
    
\begin{example} Consider $F(t)=e^{-\mid t\mid}A$, $t>0$. We write
$$
\begin{array}{llll}
\widehat{F}(\omega) &=&\displaystyle A\int_{-\infty}^{\infty}{e^{-\mid t \mid}e^{-i\omega t}dt}\\
&=&\displaystyle A\left[\int_{-\infty}^{0}e^{t(1-i\omega t)}dt \;\oplus\; \int_{0}^{\infty}e^{-t(1+i\omega t)}dt\right] \\
&=& \displaystyle \frac{2}{1+\omega^2}A.
\end{array}
$$
\end{example}

The Fourier transform can be employed to convert FPDEs into fuzzy ordinary differential equations (ODEs). To this end, the Fourier transform is applied with respect to one independent variable, thereby turning derivatives with respect to that variable into multiplication, as stated in item 5 of Proposition~\ref{prop:Prop_Fourier}.
However, in order to obtain a fuzzy ODE, the Fourier transform must commute with the partial derivative with respect to the other independent variable. This commutation can be ensured by imposing suitable regularity conditions on the solution of the corresponding FPDE. For instance, one may consider a fuzzy version of differentiation under the integral sign, by imposing assumptions analogous to those used in classical calculus on $\mathbb{R}^n$ (see, e.g., Theorem 2.27 in \cite{folland1999real}).
In contrast, the following theorem establishes weaker conditions under which the Fourier transform commutes with the weak derivative with respect to the other independent variable.

\begin{theorem}\label{thm:dif_int_sign}
Let $I = [a,b]$ and $F \in \FP(I\times \RR^n,\CFA)$ be such that $F(\cdot,\xi) \in W^{1,1}(I,\CFA)$ (or in $H^1(I,\CFA)$)  for all $\xi \in \RR^n$. 
For all $t \in I$, suppose that $F(t,\cdot),\frac{\partial}{\partial t}F(t,\cdot) \in L^1(\RR^n,\CFA)$ (or in $L^2(\RR^n,\CFA)$) and that $G(t,\omega)$ is given by the Fourier transform of $F(t,\cdot)$ at $\omega$ for all $t \in I$, that is,  
\[
G(t,\omega) = \int_{\RR^n}{F(t,\xi)e^{-i\langle \omega,\xi\rangle}d\xi}\,.
\]
If $G(\cdot,\omega) \in L^1_{loc}(I,\CFA)$, then the weak derivative of $G$ with respect to $t$ is   
\begin{equation}
\frac{\partial}{\partial t}G(t,\omega) = \int_{\RR^n}{\frac{\partial}{\partial t}F(t,\xi)e^{-i\langle \omega,\xi\rangle}d\xi}\,.
\end{equation}
for all $t \in (a,b).$
\end{theorem}
\begin{proof}
In the classical case, the proof of this result is often viewed as an immediate consequence of other results, left as an exercise, or assumed to be obvious \cite{cheng2006differentiation,schwartz1957theorie}. For the sake of clarity, we include the proof below, which follows the same idea as in the classical case.  

Given $\Phi \in C_{c}^{\infty}(I,\CFA)$. From Propositions~\ref{prop:Basic_calculus_property} and \ref{partes}, we have 
\begin{align*}
& \int_{I} \Phi'(t) \odot G(t,\omega)   \;dt  \\ 
& = \int_{I} \Phi'(t) \odot  \left(\int_{\RR^n}{F(t,\xi)e^{-i\langle \omega,\xi\rangle}d\xi} \right) \;dt\, \\ 
& = \int_{I} \int_{\RR^n}\Phi'(t) \odot F(t,\xi)e^{-i\langle \omega,\xi\rangle} \, d\xi\,dt \\
& = \int_{\RR^n} e^{-i\langle \omega,\xi\rangle} \left(\int_{I} \Phi'(t) \odot F(t,\xi) \,dt\right) \, d\xi \\
& = - \int_{\RR^n} e^{-i\langle \omega,\xi\rangle} \left(\int_{I} \Phi(t) \odot \frac{\partial}{\partial t}F(t,\xi) \,dt\right) \, d\xi \\
& = - \int_{\RR^n} \int_{I} \Phi(t) \odot \frac{\partial}{\partial t}F(t,\xi)e^{-i\langle \omega,\xi\rangle} \,dt \, d\xi \\
& = -  \int_{I} \Phi(t) \odot \left(\int_{\RR^n} \frac{\partial}{\partial t}F(t,\xi)e^{-i\langle \omega,\xi\rangle} \, d\xi \right) \,dt \,.
\end{align*}
Since $\Phi$ was taken arbitrarily, we conclude that the weak derivative of $G$ with respect to $t$ is the Fourier transform of $\frac{\partial}{\partial t}F(t,\cdot)$. 

The proof for the case where $F(t,\cdot),\frac{\partial}{\partial t}F(t,\cdot) \in L^2(\RR^n,\CFA)$ is similar, but it uses Equation~\eqref{eq:fourier_L2} instead.  
\end{proof}

It is worth noting that the condition $G(\cdot,\omega) \in L^1_{loc}(I,\CFA)$ in Theorem~\ref{thm:dif_int_sign} can be replaced by more a restrictive condition such as $G(\cdot,\omega) \in C^0(I,\CFA)$ or $G(\cdot,\omega) \in L^p(I,\CFA)$, $1\leq p < \infty$, since these spaces are subsets of $L^1_{loc}(I,\CFA)$.

This establishes the fundamental framework for Fourier analysis in the space of $A$-linearly interactive fuzzy complex processes. With these tools in place, we now turn to investigate their applications to partial differential equations, focusing specifically on the fuzzy heat equation and the fuzzy Schr\"odinger equation in the following section.

\section{Applications to Fuzzy Partial Differential Equations}

In this section, we illustrate how the theoretical framework developed in the previous sections can be applied to analyze fuzzy partial differential equations (PDEs). 
Here, for the sake of simplicity, as usual in the literature of PDE, we denote  $\frac{\partial}{\partial x}u = u_x$, $\frac{\partial}{\partial t}u = u_t$, $\frac{\partial^2}{\partial x^2}u = u_{xx}$, $\frac{\partial^2}{\partial t^2}u = u_{tt}$, and $\frac{\partial^2}{\partial x \partial t}u = u_{xt}\,.$  
Consider the following linear fuzzy differential equation:
\begin{equation}\label{eq:linear_fpde}
B\odot u_{xx} \oplus C\odot u_{tt} \oplus D\odot u_{xt} \oplus E\odot u_{x}  \oplus F\odot u_{t} = g     
\end{equation}
where $u,g \in \FP(\RR\times [0,T], \CFA)$, $u(x,t)$ is twice weakly partially differentiable, and $B,C,D,E,F \in \CFA$ with $\{0\} \neq \{B,D,E\}$ and $\{0\} \neq \{C,D,F\}$. Usually, $x$ represents the spatial variable and $t$ the time variable. 

Let us convert the fuzzy PDE~\eqref{eq:linear_fpde} into a fuzzy ODE in time space ($t$). To this end, we apply the Fourier transform in both side of Equation~\eqref{eq:linear_fpde} with respect to the spatial variable $x$.  
To ensure that the corresponding Fourier transforms exist it is enough to assume that $u_{ij}(\cdot,t),u_{i}(\cdot,t), u(\cdot,t), g$ belong to  $L^1(\RR,\CFA)$, or belong to  $L^2(\RR, \CFA)$ in the case of considering the Fourier-Plancherel transforms, for $i,j \in \{x,t\}$. 
Often, we assume that they belong to $L^2(\RR,\CFA)$, because, among other advantages of this space, we can exploit its characterization in terms of the Fourier transform (see item 4 of Theorem~\ref{thm:prop_Fourier_L2}). 
In both cases, from the linearity of the fuzzy Fourier transform, we have 
\begin{equation}\label{eq:linear_fpde_fourier}
-\omega^2 B\odot \hat{u} \oplus C \odot \widehat{u_{tt}} \oplus -i\omega D\odot \widehat{u_{t}} \oplus -i\omega E\odot \hat{u}  \oplus F\odot \widehat{u_{t}} = \widehat{g}     
\end{equation}
where $\hat{u} \equiv \hat{u}(\omega,t)$.  

To obtain a fuzzy ODE, the Fourier transform must commute with the partial derivative with respect to $t$. 
To this end, in view of Theorem~\ref{thm:dif_int_sign}, we assume additional regularity conditions on $u$. Specifically, if $C \neq 0$, then we assume $u(x,\cdot) \in H^2([0,T],\CFA)$ and $\widehat{u_{tt}}(\omega,\cdot) \in L^1_{loc}([0,T],\CFA)$; if $C = 0$, we assume that $u(x,\cdot) \in L^2([0,T],\CFA).$ Moreover, we assume that $\widehat{u_{t}}(\omega,\cdot) \in L^1_{loc}([0,T],\CFA)$ if $D\neq 0$ or $F\neq 0$.
Under these conditions, we obtain the following linear fuzzy ordinary differential equation:
\begin{align}\label{eq:linear_ode_fourier} \nonumber
& -\omega^2 B\odot \hat{u} \oplus C \odot \widehat{u}_{tt} \oplus -i\omega D\odot \widehat{u}_{t} \oplus -i\omega E\odot \hat{u}  \oplus F\odot \widehat{u}_{t} = \widehat{g} \\
\Rightarrow & C \odot \widehat{u}_{tt} \oplus \left(-i\omega D\oplus F\right)\odot \widehat{u}_{t}  \oplus  \left(-\omega^2 B  \oplus -i\omega E\right)\odot \hat{u}  = \widehat{g}\,.
\end{align}

When initial conditions and/or boundary conditions are given, we may impose similar regularity assumptions on them to apply the Fourier transform. This, in turn, may require additional hypotheses on the solution $u$. 
To illustrate the use of the Fourier transform in solving FPDEs within the present framework, we focus on the fuzzy versions of the classical heat and  Schr\"odinger equations. These fundamental equations of mathematical physics can be adapted to the fuzzy setting to model diffusion and wave-like phenomena under uncertainty.

To solve these fuzzy PDEs, we adopt a decomposition strategy that leverages the linear structure of $\CFA$. Both the initial condition and the solution are expressed in the form
\[
F(x) = r_F(x) + q_F(x)A, \quad \text{and} \quad u(x,t) = r(x,t) + q(x,t)A,
\]
where $r_F$, $q_F$, $r$, and $q$ are real-valued functions in heat equation and complex-valued functions in Schr\"odinger equation. 

It is worth noting that these equations have studied in the frameworks of $\RFA$ and $\CFA$. In \cite{de2021differential}, the authors consider fuzzy heat equations with real parameters. The case with fuzzy parameters in $\RFA$ was addressed in \cite{shen2025solutions} and \cite{Mina}. However, the author of  \cite{shen2025solutions} consider the $A$-cross product, whereas the authors in \cite{Mina} employ a general class of multiplication in $\RFA$, called distributive product. The fuzzy Schr\"odinger equation in the $\CFA$ setting was studied in \cite{Salgado}. In contrast to these previous works, our aim here is to illustrate the use of the Fourier transform and to analyze the regularity of solutions of these fuzzy PDEs. For the sake of simplicity, we focus on the real coefficients. This representation allows us to separate the fuzzy PDE into two classical PDEs for the components $r$ and $q$, which are then treated using classical Fourier analysis.

\subsection{Fuzzy Heat Equation}

The heat equation (or thermal diffusion equation) describes how heat propagates through a medium over time. We consider the follwowing problem for the fuzzy heat equation. Given $F:\mathbb{R} \longrightarrow \CFA$, determine $u: \RR \times [0,t] \longrightarrow \CFA$, such that 
\begin{equation}\label{heat}
\left\{
\begin{array}{llll}
u_t=ku_{xx}\\
u(x,0)=F(x)
\end{array}
\right.
\end{equation}
where $k \in \RR.$

Suppose that the solution $u$ is such that $u(\cdot,t) \in H^2(\mathbb{R}, \CFA)$ for every $t$. Will also assume that $u(x,\cdot) \in L^2([0,T], \CFA)$ and $\widehat{u_{t}}(\omega,\cdot) \in L^1_{loc}([0,T],\CFA).$ Finally, for the initial condition to make sense,we require that  
$F \in L^2(\mathbb{R}, \CFA)$.

Considering $F(x)=r_{F}(x)+q_{F}(x)A$ with $r_{F},q_{F}: \mathbb{R} \longrightarrow \mathbb{C}$ and $u(x,t)=r(x,t)+q(x,t)A$, where $r,q: \mathbb{R} \times [0,t) \longrightarrow \mathbb{C}$ we have:

\begin{equation}\label{q}
\left\{
\begin{array}{llll} 
q_t=kq_{xx}\\
q(x,0)=q_F(x)
\end{array}
\right.
\end{equation} 

\noindent and \begin{equation}\label{r}
\left\{
\begin{array}{llll}
r_t=kr_{xx}\\
r(x,0)=r_F(x).
\end{array}
\right.
\end{equation}
Applying the Fourier transform to the Equation (\ref{q}), we have

\begin{equation}
\widehat{q}_{t}(\xi,t)=k\widehat{q}_{xx}(\xi,t)
\end{equation}

\noindent i.e,

\begin{equation} \label{two}
\widehat{q}_{t}=k(i\xi)^2\widehat{q}=-k\xi^2\widehat{q}.
\end{equation}

Solving (\ref{two}) and using the initial condition $q(x,0)=q_{F}(x)$,

\begin{equation}
\widehat{q}(\xi,t)=\widehat{q}_{F}(\xi)e^{-k\xi^2t}.
\end{equation}

Therefore, 
\begin{equation*}
\begin{array}{lll}
q(x,t)&=&\displaystyle \frac{1}{2 \pi}\int_{\mathbb{R}}e^{ix \xi}\widehat{q}(\xi,t)d\xi\\
&=&\displaystyle \frac{1}{2 \pi}\int_{-\infty}^{+\infty}e^{ix \xi}\widehat{q}_{F}(\xi)e^{-k\xi^2t}d \xi\\
&=&\displaystyle \frac{1}{2 \pi}\int_{-\infty}^{+\infty}e^{ix \xi}\left[\int_{-\infty}^{+\infty}e^{-iy \xi}q_{F}(y)dy\right]e^{-k\xi^2t}d \xi\\
&=&\displaystyle \frac{1}{2 \pi}\int_{-\infty}^{+\infty}q_{F}(y)\left[\int_{-\infty}^{+\infty}e^{-i(y-x) \xi}e^{-k\xi^2t}d \xi\right]dy\\
&=&\displaystyle \frac{1}{2 \pi}\int_{-\infty}^{+\infty}q_{F}(y)\widehat{f}(y-x)dy\\
\end{array}
\end{equation*}
where $f(\xi)=e^{-k\xi^2t}$. If $f(\xi)=e^{-k\xi^2t}$ implies $\widehat{f}(z)=\displaystyle \sqrt{\frac{\pi}{kt}}e^{-\frac{(x-y)^2}{4kt}}$, and consequently, the solution of the initial-value problem (\ref{q}) is given by

\begin{equation}
q(x,t)=\displaystyle \frac{1}{\sqrt{4\pi kt }}\int_{-\infty}^{+\infty}q_{F}(y)e^{-\frac{(x-y)^2}{4kt}}dy
\end{equation}
for $t>0$.

Similarly, the solution of the initial-value problem (\ref{r}) is given by

\begin{equation}
r(x,t)=\displaystyle \frac{1}{\sqrt{4\pi kt }}\int_{-\infty}^{+\infty}r_{F}(y)e^{-\frac{(x-y)^2}{4kt}}dy
\end{equation}
for $t>0$.

Therefore, the solution of the (\ref{heat}) is given by
\begin{equation}\label{sol}
u(x,t)= \displaystyle \frac{1}{\sqrt{4\pi kt }}\int_{-\infty}^{+\infty}r_{F}(y)e^{-\frac{(x-y)^2}{4kt}}dy+\left(\displaystyle \frac{1}{\sqrt{4\pi kt }}\int_{-\infty}^{+\infty}q_{F}(y)e^{-\frac{(x-y)^2}{4kt}}dy\right)A
\end{equation}
for $t>0$.

Hence, if $F \in L^2(\RR, \CFA)$, the heat equation \eqref{heat} admits a unique solution $u$ that satisfies the hypotheses above: $u(\cdot,t) \in H^2(\mathbb{R}, \CFA)$ for every $t$ and $u(x,\cdot) \in L^2([0,T], \CFA)$ and $\widehat{u_{t}}(\omega,\cdot) \in L^1_{loc}([0,T],\CFA).$
One can easi\-ly verify that the solution \eqref{sol} satisfies all these requirements. In fact, one can show that the solution $u$  satisfies a more restrictive property: the mapping $t\mapsto u(\cdot,t)$ is a continuous operator from $[0,T]$ into $H^2(\RR, \CFA)$ and a continuously (Fr\'echet) differentiable from $(0,T)$ into $L^2(\mathbb{R}, \CFA).$ This is often denoted by $u \in C^0([0,T]; H^2(\RR, \CFA) \cap 
C^1([0,T]; L^2(\RR, \CFA).$

\subsection{Fuzzy Schr\"odinger Equation}

The Schr\"odinger equation is the fundamental equation of non-relativistic quantum mechanics, describing the time evolution of quantum systems.
Given  $F:\mathbb{R} \longrightarrow \CFA$, the fuzzy Schr\"odinger equation is 
\begin{equation}\label{SC}
\left\{
\begin{array}{llll}
u_{xx}=-iu_{t}\\
u(x,0)=F(x)
\end{array}
\right.
\end{equation}
where $u: \mathbb{R} \times [0,T] \longrightarrow \CFA$.  

Considering $F(x)=r_{F}(x)+q_{F}(x)A$ with $r_{F},q_{F}: \mathbb{R} \longrightarrow \mathbb{C}$ and $u(x,t)=r(x,t)+q(x,t)A$, where $r,q: \mathbb{R} \times [0,T] \longrightarrow \mathbb{C}$, we have:
\begin{equation}\label{q1}
\left\{
\begin{array}{llll} 
q_{xx}=-iq_{t}\\
q(x,0)=q_F(x)
\end{array}
\right.
\end{equation} 
\noindent and 
\begin{equation}\label{r1}
\left\{
\begin{array}{llll}
r_{xx}=-ir_{t}\\
r(x,0)=r_F(x).
\end{array}
\right.
\end{equation}
Applying the Fourier transform to the Equation~\eqref{q1}, we obtain 
\begin{equation}
-x^2\widehat{q}(x,t)=-i\displaystyle \frac{\partial }{\partial t} \widehat{q}(x,t).
\end{equation}
So, 
\begin{equation}
\displaystyle \frac{\partial }{\partial t} \widehat{q}(x,t)=\frac{x^2}{i} \widehat{q}(x,t)=-ix^2\widehat{q}(x,t)
\end{equation}
i.e,
\begin{equation}
\widehat{q}(x,t)=e^{-ix^2t}c(x)
\end{equation}
where $c(x)=\widehat{q}(x,0)=\widehat{q_{F}}(x).$
Since we know $\widehat{q}(x,t)$, by the Fourier inversion formula we have:
\begin{equation}
\begin{array}{lllllllllllll}
q(x,t)&=&\displaystyle \frac{1}{2\pi}\int_{-\infty}^{+\infty}e^{ixy}\widehat{q}(y,t)dy\\
&=&\displaystyle \frac{1}{2\pi}\int_{-\infty}^{+\infty}e^{ixy}\left(e^{-iy^2t}\widehat{q_{F}}(y)\right)dy.
\end{array}
\end{equation}
Similarly, the solution of the initial-value problem (\ref{r1}) is given by
\begin{equation}
r(x,t)=\displaystyle \frac{1}{2\pi}\int_{-\infty}^{+\infty}e^{ixy}\left(e^{-iy^2t}\widehat{r_{F}}(y)\right)dy.
\end{equation}
Therefore, the solution of the (\ref{SC}) is given by
\begin{equation}\label{sol1}
u(x,t)=\displaystyle \frac{1}{2\pi}\int_{-\infty}^{+\infty}e^{ixy}\left(e^{-iy^2t}\widehat{r_{F}}(y)\right)dy+\left(\displaystyle \frac{1}{2\pi}\int_{-\infty}^{+\infty}e^{ixy}\left(e^{-iy^2t}\widehat{q_{F}}(y)\right)dy\right)A.
\end{equation}

Thus, for $F \in L^2(\RR, \CFA)$, the heat equation~\eqref{heat} admits a unique solution $u$ that satisfies the hypotheses: $u(\cdot,t) \in H^2(\mathbb{R}, \CFA)$ for every $t$ and $u(x,\cdot) \in L^2([0,T], \CFA)$ and $\widehat{u_{t}}(\omega,\cdot) \in L^1_{loc}([0,T],\CFA).$ Again,  one can show that the solution can be viewed as mapping $t\mapsto u(\cdot,t)$ is such that 
$$u \in C^0([0,T]; H^2(\RR, \CFA) \cap C^1([0,T]; L^2(\RR, \CFA).$$

\section{Final Remarks}

In this study, we introduced fuzzy Sobolev spaces and the fuzzy Fourier transform for $A$-linearly interactive fuzzy complex processes. The main motivation for developing these concepts lies in the theory of fuzzy partial differential equations for these class of fuzzy functions. In Section~\ref{sec:sobolev}, we defined the fuzzy Lebesgue space $L^p(\Omega,\CFA)$ and proved that it is a Banach space. Furthermore, we introduced the concept of fuzzy partial derivative in the weak sense, which allowed us to introduce the fuzzy Sobolev space $W^{k,p}(\Omega,\CFA)$ and show that it also forms a Banach space. The notion of Fourier transform was introduced for fuzzy spaces $L^1$ and $L^2$ and we prove several of its properties in Section~\ref{sec:Fourier}. Similar to classical case, the fuzzy Sobolev spaces based on $L^2$ can be completely characterized in terms of the fuzzy Fourier transform of their elements (see item 4 of Theorem~\ref{thm:prop_Fourier_L2}). The proposed framework can be used not only to solve but also to study the regularity of solutions of linear fuzzy PDEs. To illustrate the usefulness, we investigate explicit solutions of some fuzzy partial differential equations, such as the fuzzy heat equation and the fuzzy Schr\"odinger equation.

Our main goal is to address problems whose final solutions are  fuzzy number-valued functions, even though intermediate steps of the method (such as the application of the Laplace or Fourier transforms) may involve complex-valued expressions. This is analogous to the deterministic case, in which the solution of a real-valued differential equation often passes through the complex domain, using exponential functions and subsequently recovering real solutions via sine and cosine functions. 
One of the advantages of working with the spaces $\RFA$ or $\CFA$ is that, in addition to extending the sets of real and complex numbers, they also form Banach spaces in which $\RR$ and $\CC$ are vector subspaces, respectively.
As a consequence, many classical results can be naturally extended to these spaces. In particular, this is also the case for fuzzy Lebesgue and Sobolev spaces, for which we show that they are linearly homeomorphic to the Cartesian products of the corresponding classical Lebesgue and Sobolev spaces (see Theorems~\ref{thm:charaterize_Lp} and \ref{thm:characterization_Sobolev}, respectively). As a result, a fuzzy PDE can be translated into a classical system of PDEs via this linear homeomorphism, thereby providing a suitable framework for solving fuzzy PDEs analytically and numerically by means of classical methods. 

\section*{Acknowledgements}

This research was partially supported by FAPEMIG under grant no. APQ-02561-25 (S.A.B. Salgado), by FAPESP under grants no. 2020/09838-0 (E. Esmi), and by CNPq under grants no. 311976/2023-9 (E. Esmi) and 314885/2021-8 (L.C. Barros).


\end{document}